\documentclass[10pt]{article}
\usepackage{amsmath,amsfonts,amssymb,amsthm,fancyhdr,bm,mathtools,enumitem,bbm,stmaryrd}
\usepackage[a4paper, margin=2 cm]{geometry}
\usepackage[hyphens]{url}
\usepackage[hidelinks]{hyperref}
\usepackage{fullpage}
\usepackage[dvipsnames]{xcolor}
\usepackage[capitalize]{cleveref}
\usepackage[color=Purple!50!white,textsize=tiny]{todonotes}
\usepackage[numbers,sort]{natbib}
\usepackage{mathtools}
\usepackage{tikz,ifthen}
\usepackage{comment}
\usepackage{dsfont}

\usepackage{dingbat}
\usepackage{float}
\usepackage{graphicx}
\usepackage{caption}

\usetikzlibrary{decorations.markings,arrows.meta}
\tikzset{vert/.style={draw, fill=black, circle, inner sep=2pt}}

\newtheorem{theorem}{Theorem}[section]
\newtheorem{lemma}[theorem]{Lemma}

\newtheorem{conjecture}[theorem]{Conjecture}
\newtheorem{observation}[theorem]{Observation}

\newtheorem{claim}[theorem]{Claim}
\newtheorem{subclaim}[theorem]{Sub-Claim}
\newtheorem{fact}[theorem]{Fact}
\theoremstyle{definition}
\newtheorem{definition}[theorem]{Definition}

\crefname{equation}{equation}{equations}
\crefname{lemma}{Lemma}{Lemmas}
\crefname{proposition}{Proposition}{Propositions}
\crefname{claim}{Claim}{Claims}
\crefname{theorem}{Theorem}{Theorems}
\crefname{conjecture}{Conjecture}{Conjectures}
\crefname{figure}{Figure}{Figures}

\AddToHook{env/lemma/begin}{\crefalias{theorem}{lemma}}
\AddToHook{env/proposition/begin}{\crefalias{theorem}{proposition}}
\AddToHook{env/corollary/begin}{\crefalias{theorem}{corollary}}
\AddToHook{env/conjecture/begin}{\crefalias{theorem}{conjecture}}
\AddToHook{env/claim/begin}{\crefalias{theorem}{claim}}
\AddToHook{env/question/begin}{\crefalias{theorem}{question}} 
\AddToHook{env/definition/begin}{\crefalias{theorem}{definition}} 
\AddToHook{env/remark/begin}{\crefalias{theorem}{remark}} 

\newlist{lemenum}{enumerate}{1}
\setlist[lemenum]{label=(\alph*), ref=\thelemma(\alph*)}
\crefalias{lemenumi}{lemma}

\newcommand\ab[1]{\lvert#1\rvert}
\newcommand\bab[1]{\left\lvert#1\right\rvert}

\newcommand\eps{\varepsilon}

\let\leq\leqslant
\let\geq\geqslant

\title{Robustness and hyperstability for the Erd\H{o}s--Gallai theorem}
\author{Micha Christoph\thanks{Department of Mathematics, ETH Z\"{u}rich, 8006 Z\"urich, Switzerland. Email: \texttt{micha.christoph@math.ethz.ch}. Research supported by SNSF Ambizione Grant No. 216071.}\and Alp M\"uyesser\thanks{Mathematical Institute, University of Oxford, OX2 6GG, UK. Email: \texttt{alp.muyesser@maths.ox.ac.uk}} \and Yuval Wigderson\thanks{Institute for Theoretical Studies, ETH Z\"urich, 8006 Z\"urich, Switzerland. Email: \texttt{yuval.wigderson@eth-its.ethz.ch}. Research supported by Dr.\ Max R\"ossler, the Walter Haefner Foundation, and the ETH Z\"urich Foundation.}}
\date{}

\begin{document}

\maketitle
\begin{abstract}
    The Erd\H{o}s--Gallai theorem states that every graph of average degree $d$ contains a cycle of length at least $d$. We prove the following robust extension of the Erd\H{o}s--Gallai theorem: For every $c>0$ there exists $K$ such that for all $d\geq K$, $p\geq K/d$ and every graph $G$ with average degree $d$, the random graph $G_p$ obtained by independently sampling each edge of $G$ with probability $p$ contains a cycle of length at least $(1-c)d$ asymptotically almost surely as $|V(G)|\to \infty$. 
    With related methods, we prove the following hyperstability version of the Erd\H{o}s--Gallai theorem: any graph $G$ without a cycle of length at least $d$ is at most $c d|V(G)|$ edge deletions away from a graph all of whose connected components have a vertex-cover of size at most $d$.
    \par At the core of our argument lies a very general structure theorem about graphs that originates from results of Pokrovskiy concerning the hyperstability of bounded-degree trees. 
\end{abstract}
\section{Introduction}

One of the foundational results in extremal graph theory is the Erd\H os--Gallai \cite{MR114772} theorem, which states that if a graph $G$ has average degree $d\geq 2$, then it contains a cycle of length at least $\lfloor d \rfloor+1\geq d$ (and hence also a path of length at least $\lfloor d \rfloor$). This is easily seen to be essentially best possible, by taking $G$ to be a disjoint union of copies of $K_{d+1}$. 

Despite its innocuously simple statement, the Erd\H os--Gallai theorem is quite substantial, for a number of reasons. Firstly, unlike most other classical theorems of extremal graph theory (e.g.\ Tur\'an's theorem, the Erd\H os--Stone--Simonovits theorem, or the K\H ov\'ari--S\'os--Tur\'an theorem), the Erd\H os--Gallai theorem does not tell us \emph{which} subgraph we are guaranteed to find: it implies the existence of a cycle of length at least $d$, but not a cycle of any given fixed length. As is well-known, this feature is unavoidable in such a statement: there exist graphs of average degree $d$ but of arbitrarily large girth, hence it is impossible to guarantee any fixed cycle length in the setting of the Erd\H os--Gallai theorem.  
Secondly, both the statement and the proof of the Erd\H os--Gallai theorem are closely related to the study of Hamiltonicity, which is one of the most challenging and interesting areas in modern graph theory; indeed, the original proof of Erd\H os and Gallai used as a black box a result of Dirac \cite[Theorem 4]{MR47308} proved in his famous work on Hamiltonicity of high-degree graphs, and to the best of our knowledge, every known proof of the Erd\H os--Gallai theorem relies on such a result. Finally, the set of extremal examples in the Erd\H os--Gallai theorem is rather complicated. As a consequence, the optimal extremal result (obtained independently by Kopylov \cite{MR463030} and Woodall \cite{MR427143}) is quite difficult to prove. Moreover, unlike other extremal problems where the extremal structure is quite ``robust'' and exhibits strong stability properties, the stability theory for the Erd\H os--Gallai theorem is rather delicate, and the best known results \cite{MR3777040,MR3548292,MR4078813} only hold when the number of edges is extremely close to the maximum possible number. For more information on these points, we refer to \cite[Section 5]{MR3203598}.

Another related, and extremely influential, line of research concerns long paths and cycles in random graphs. Recall that the seminal work of Erd\H os and R\'enyi \cite{MR125031} demonstrates that the binomial random graph $G(n,p)$ exhibits a remarkable phase transition around $p=1/n$: if $p\leq (1-\delta)/n$ for any fixed $\delta>0$, then asymptotically almost surely $G(n,p)$ contains no component with more than $O_\delta (\log n)$ vertices, and in particular contains no path or cycle of more than that length. However, if $p \geq (1+\delta)/n$, then $G(n,p)$ typically contains a unique \emph{giant component} of order $\Omega_\delta(n)$, so it is natural to wonder whether it also contains a path or cycle of linear length. This question was one of the central open problems in the study of random graphs for over a decade, until it was finally resolved in the affirmative by Ajtai, Koml\'os, and Szemer\'edi \cite{MR602411}. Moreover, Fernandez de la Vega \cite{MR685900} showed that for every $\varepsilon>0$, there exists some $K=K(\varepsilon)>0$, such that $G(n,K/n)$ asymptotically almost surely contains a very long cycle, namely one of length at least $(1-\varepsilon)n$. The optimal dependence of $K$ on $\varepsilon$ was later resolved by Frieze \cite{MR842277}, building on earlier work of \cite{bollobas1984long,MR698649}.

In recent decades, there has been a great deal of interest in merging the classical studies of extremal graph theory and random graph theory. While there are a huge number of variants of such questions and results, including the randomly perturbed model \cite{bohman2003many}, local and global resilience of graph properties \cite{lee2012dirac}, and the study of general transference principles \cite{conlon2016combinatorial,MR3548528}, we focus here on a natural and well-studied notion called \emph{robustness}, and refer to Sudakov's survey \cite{MR3728112} (see also \cite{krivelevich2014robust, kelly2024optimal}) for more information and related concepts. The meta-question underlying the study of robustness is the following: given $p\in[0,1]$ and a graph $G$ with a certain property, does the property still hold almost always if we sample (or percolate) each edge of $G$ with probability $p$? Concretely, if we start with a graph $G$ of average degree $d$, and keep each edge randomly and independently with probability $p$, can we expect the resulting random subgraph $G_p$ to still have a cycle of length at least about $d$? 

Our first main result addresses this question. In particular, we settle a problem raised by Glebov, Naves, and Sudakov \cite[Section 4]{glebov2017threshold}.
\begin{theorem}\label{thm: main}
    For every $c>0$ there exists $K$ such that the following holds for all $d\geq K$ and $p\geq K/d$. Let $G$ be a graph on $n$ vertices with average degree $d$. Then, with probability tending to $1$ as $n\to \infty$, $G_p$ contains a cycle of length at least $(1-c)d$.

    In particular, if $G$ has average degree $d$ and $p = \omega(1/d)$, then $G_p$ contains a cycle of length at least $(1-o(1))d$ asymptotically almost surely.
\end{theorem}
Note that, by taking $p=1$, we asymptotically recover the Erd\H os--Gallai theorem, and by taking $G=K_{d+1}$, we recover the results discussed above about long cycles in binomial random graphs. Note too that the condition $p \geq K/d$ is best possible up to the value of $K$; for instance, if $G=K_{d+1}$, then roughly a $Ke^{-K}$-fraction of the vertices of $G_{K/d}$ have degree at most $1$, and hence any cycle necessarily has length at most $(1-\varepsilon)d$, for some constant $\varepsilon>0$. 

\cref{thm: main} is a substantial strengthening of a result of 
Krivelevich, Lee, and Sudakov \cite{MR3302900}, who were the first to study such questions. They proved a result like \cref{thm: main}, but under the much stronger assumption that the \emph{minimum} degree of $G$ is $d$, rather than the \emph{average} degree\footnote{Their result is also weaker than ours in that they only obtain a probability tending to $1$ as $d \to \infty$, rather than as $n \to \infty$. However, this is a much less significant weakness than the assumption of minimum degree, and it moreover seems likely that their proof could be modified to overcome this.}. A simpler proof of the same result was subsequently given by Riordan \cite{MR3275705}. Both proofs rely on the \emph{depth-first search} technique, which is also an important component of our proof, but they also both seem to crucially use the assumption of minimum degree. This is closely related to the fact that the Erd\H os--Gallai theorem is highly non-trivial, whereas it is an elementary exercise to prove that a graph of minimum degree $d$ contains a cycle of length at least $d+1$. 

\cref{thm: main} can be viewed as a way of saying that the conclusion of the Erd\H os--Gallai theorem is robust with respect to random edge deletions. Our next main result also demonstrates a certain form of robustness, but in this case it is robustness of the extremal examples. In order to explain what this means, let us recall that an example witnessing the tightness of the Erd\H os--Gallai theorem is a disjoint union of copies of $K_{d+1}$, but that this is far from the only example; for instance, a complete bipartite graph $K_{d/2,n}$ has average degree $(1-o(1))d$ as $n \to \infty$, but certainly contains no cycle of length more than $d$. There are other extremal and near-extremal examples (see e.g.\ \cite[Construction 5.3]{MR3203598}), but in all instances the inherent reason why there is no cycle of length more than $d$ appears to be that every connected component\footnote{More precisely, every block, that is, $2$-connected component.} contains a small vertex-cover, which witnesses that this component cannot contain a long cycle. The upshot of our second main theorem will be an explanation of this phenomenon, but before stating it, let us discuss stability a bit more broadly, to put our result into its proper context.

Generally, in extremal combinatorics, once one has proven an extremal statement (and identified the corresponding extremal structure), the natural next step is to prove \emph{stability}: if an object is close to the extremal size satisfying some property, then it should also be {close} (in some appropriate sense) to the extremal object. In addition to being interesting in their own right, such stability statements are often extremely useful in proving other results; we refer to the survey \cite[Section 5]{MR2866732} for more details on the so-called \emph{stability method} and its myriad applications. For many extremal problems, such as Tur\'an's theorem, we have access to strong stability statements, that prove structural similarity for graphs which contain a $(1-\varepsilon)$-fraction of the edges of the unique extremal example. However, for the Erd\H os--Gallai theorem, we only have access to much weaker results \cite{MR4078813,MR3548292,MR3777040}, e.g.\ only obtaining structural information for graphs that are within an additive $O(1)$ of the extremal average degree.

In additive combinatorics, such stability statements are usually called \emph{inverse theorems}, and it is common to distinguish between so-called 99\% inverse theorems and 1\% inverse theorems (see e.g.\ \cite[Section 1.5]{MR2931680}). In the former, one assumes that some quantity is very close to its extremal value, and hopes to deduce strong structural properties. In the latter, one instead only assumes that our quantity is a small constant fraction of the extremal value, and still hopes to obtain some (necessarily weak) structural information; usually the best one can hope for is some small amount of ``correlation'' with an extremal example. In general, the stability results common in extremal graph theory are 99\% statements: while we understand quite well what a triangle-free graph with $(\frac 14-\varepsilon)n^2$ edges should look like (it is close to a subgraph of the unique extremal structure, the complete bipartite graph $K_{n/2,n/2}$), there is essentially nothing interesting that can be said about a triangle-free graph with only $\varepsilon n^2$ edges. In extremal graph theory, there are essentially no meaningful 1\% stability theorems that are known. 

Our next main result demonstrates that for the Erd\H os--Gallai theorem, we have a ``hyperstability'' phenomenon, that captures the strongest parts of both the 99\% and the 1\% stability results.
\begin{theorem}\label{thm:EG hyperstability}
For every $c>0$ and every sufficiently large $d$, the following holds for every $n$-vertex graph $G$. Either $G$ contains a cycle of length at least $d$ or one can delete at most $c dn$ edges from $G$ to obtain a graph in which every connected component has a vertex-cover of order at most $d$. 
\end{theorem}
Note that, in contrast to the Erd\H os--Gallai theorem, we make \emph{no assumptions} on the average degree of $G$; that being said, the theorem is only interesting if $G$ contains more than $cdn$ edges, for otherwise we can delete all edges of $G$ in the second case. Note also that the theorem is asymptotically best possible since a disjoint union of copies of $K_{d-1}$ has no cycle of length at least $d$, and, no matter how one deletes $o(dn)$ edges, still yields a graph with a component whose smallest vertex-cover is of order $(1-o(1))d$. Thus, as promised, \cref{thm:EG hyperstability} captures the strongest aspects of both the 99\% and the 1\% stability regimes: the assumptions of the theorem are of the 1\% variety (since it yields a sensible statement even when the average degree of $G$ is an $\varepsilon$-fraction of the maximum possible, as given in the Erd\H os--Gallai theorem), but the conclusion is of the 99\% variety (since it essentially states that, up to a small amount of noise, $G$ is contained in a disjoint union of graphs with vertex-cover $(1+o(1))d$, as in the extremal example). We stress that while our theorem is not formally stronger than the other stability results for the Erd\H os--Gallai theorem \cite{MR4078813,MR3548292,MR3777040} (they give exact results and \cref{thm:EG hyperstability} is approximate), our \cref{thm:EG hyperstability} is applicable in a \emph{far} greater range of settings. 


We are aware of only one other example of such a hyperstability theorem in extremal combinatorics, namely a recent result of Pokrovskiy \cite{alexey}. Pokrovskiy proved a variant of \cref{thm:EG hyperstability}, where instead of avoiding a cycle of length at least $d$, $G$ instead avoids a copy of some bounded-degree $d$-vertex tree\footnote{Another difference is that he only obtains a cover of size $(2+o(1))d$, which is not tight in general; however, for many natural classes of trees, such as paths, his proof can be easily modified to yield a cover of size $(1+o(1))d$.}. In addition to inspiring the statement of \cref{thm:EG hyperstability}, Pokrovskiy's work also forms the crucial technical underpinning of both of our main results.

Indeed, the proofs of both \cref{thm: main,thm:EG hyperstability} rely heavily on a remarkable structure theorem, which yields an approximate decomposition of \emph{any} graph into four types of subgraphs. While each type is somewhat complex to understand, they are all amenable to study by some fairly standard techniques in extremal graph theory. For example, one of the types is a strong (local) expander, and there are many techniques for finding long cycles in expanders; the other three types are dense graphs overlapping in certain prescribed ways, for whose study one has access to powerful tools like Szemer\'edi's regularity lemma. While this structure theorem does not explicitly appear in \cite{alexey}, it is implicit in Pokrovskiy's work; one of the contributions of the present paper is to isolate it in as ``user-friendly'' a way as possible, which we do in Section~\ref{sec:pokrovskiy}.

For now, we merely state an informal version of Pokrovskiy's remarkable structure theorem. 

\begin{theorem}[Informal version of Theorem~\ref{thm: structure}]\label{thm: structure informal}
    Let $C>0$ and let $d$ be sufficiently large. For every $n$-vertex graph $G$,
    there is an edge-decomposition of $G$ into $G_0,G_1,G_2,G_3,G_4$ satisfying the following properties.
    \begin{description}
        \item[Exceptional edges] The graph $G_0$ contains at most $o(e(G))$ edges.
        \item[Small vertex-cover] Every component of $G_1$ has a vertex-cover of size at most $Cd$.
        \item[Large regular cores] $G_2$ is the union of edge-disjoint $\Omega(1)$-cut-dense\footnote{A graph $H$ is said to be $\rho$-cut-dense if, for all partitions $V(H)=A\sqcup B$, we have $e_H(A,B)\geq \rho\ab A \ab B$.} graphs $H$, each of which contains a $k$-regular subgraph of order at least $Cd$, where $k=\Theta(d)$ and $\ab{V(H)}= O(Cd)$.  
        \item[Nowhere dense] $G_3$ does not contain $O(d)$ vertices inducing $\Omega(d^2)$ edges. 
        \item[Well-connected pieces] $G_4$ is the edge-disjoint union of a collection $\mathcal H$ of graphs satisfying the following properties.
        \begin{enumerate}
            \item Each $H\in\mathcal H$ is $\Omega(d)$-connected and has diameter $O(1)$.
            \item $|H\cap H'|=o(d)$ for all $H,H' \in \mathcal H$.
            \item Every $H\in \mathcal H$ has $\Omega(d)$ vertices which intersect some other $H'\in\mathcal H$.
        \end{enumerate}
        
    \end{description}
\end{theorem}
\noindent We note that there is no assumption on $G$ in \cref{thm: structure informal}.

In order to help the reader digest the statement of \cref{thm: structure informal}, we now describe how to use this theorem in practice (the reader might also benefit from Section~\ref{sec:necessity} where we discuss why each of $G_1,G_2,G_3$ and $G_4$ is crucial to include in the statement). Specifically, we will sketch how to asymptotically deduce the Erd\H os--Gallai theorem from \cref{thm: structure informal}; while this proof is several orders of magnitude more complicated than the original proof of Erd\H os and Gallai, we think it is also instructive, as the proofs of both \cref{thm: main,thm:EG hyperstability} will follow essentially the same framework (recall that each of our main theorems is a strengthening of the Erd\H os--Gallai theorem). 

Concretely, let us fix $c>0$, and let $G$ be an $n$-vertex graph of average degree $d$. Our goal is to show that if $d$ is sufficiently large, then $G$ contains a cycle of length $(1-c)d$. Let us apply \cref{thm: structure informal} to $G$ with parameter $C=10$, to obtain an edge-partition $G=G_0 \cup G_1 \cup G_2 \cup G_3 \cup G_4$. Our strategy goes as follows: First, we show that each of $G_2,G_3$ and $G_4$ must be essentially empty if $G$ does not contain a cycle of length at least $d$. Then, it remains to address the case that $G_1$ contains almost all edges of $G$.

First, suppose that $G_2$ is non-empty, and fix one graph $H$ making up $G_2$. By definition (recalling that we chose $C=10$), we see that $H$ is $\Omega(1)$-cut-dense and has a $k$-regular subgraph $R$ of order at least $10d$, for some $k=\Theta(d)$. Let $V_{1/2}\subseteq V(H)$ be chosen by including every vertex with probability $1/2$ independently of all other vertices. First, Szemer\'edi's regularity lemma (see Lemma~\ref{lem: regularity}), together with the assumption that $R$ is $k$-regular, implies that, with high probability, we can find a collection $\mathcal P$ of $O(1)$ vertex-disjoint paths in $R[V_{1/2}]$ which collectively cover at least $(5-o(1))d$ vertices. Second, again using Szemer\'edi's regularity lemma, but this time together with the fact that $H$ is $\Omega(1)$-cut-dense, we conclude, with high probability, that we can connect the paths of $\mathcal P$ into a cycle through $H-V_{1/2}$. Since the total number of vertices in $\mathcal P$ is $(5-o(1))d$, we have found a cycle of length at least $(5-o(1))d$. We may thus assume that $G_2=\varnothing$.

Next, suppose that $e(G_3)\geq \eta dn/2$, so that $G_3$ has average degree at least $\eta d$, for some small $\eta>0$. Let $G_3'$ be a subgraph of $G_3$ of minimum degree at least $\eta d/2$, and note that $G_3'$ is nowhere dense just as $G_3$. Let us define a \emph{broom} to be a path one of whose endpoints has at least $\eta d/4$ neighbors not on the path. An elementary argument shows that a broom of maximal length in $G_3'$ yields a cycle of length $\omega(d)$. We may thus assume that $e(G_3)< \eta dn/2$.

Finally, let us suppose that $G_4$ is non-empty, and let $\mathcal H$ be the family of graphs comprising $G_4$. Let $F$ be the intersection graph of the elements of $\mathcal H$, that is, we put $(H,H') \in E(F)$ if and only if $V(H) \cap V(H') \neq \varnothing$. Since each $H$ in $\mathcal H$ has at least $\Omega(d)$ vertices intersecting some other $H'$, and since every such intersection contains at most $o(d)$ vertices, we see that the minimum degree of $F$ is $\omega(1)$. As a consequence, $F$ contains a cycle of length $\omega(1)$, so we may fix a cycle $H_1,H_2,\dots,H_t$ in $F$, where $t$ is at least some huge constant. As $(H_i,H_{i+1})$ is an edge of $F$ for all $i$ (with indices taken mod $t$), we may fix distinct\footnote{It is not trivial to ensure that these are distinct, a subtlety we ignore in this sketch.} vertices $v_i \in V(H_i)\cap V(H_{i+1})$. Now, since each $H_i$ is $\Omega(d)$-connected, we may pick a path $P_i$ in $H_i$ joining $v_{i-1}$ to $v_i$ (with indices mod $t$), and ensure that the length of $P_i$ is at least $\Omega(d)$, using the well-known fact that in every $k$-connected graph, every pair of vertices is joined by a path of length at least $k$. Moreover, we may pick these paths to be internally vertex-disjoint: having picked $P_1,\dots,P_i$, by the assumption that every pair of graphs in $H$ intersect in $o(d)$ vertices, we have still only used up $o(d)$ vertices\footnote{Again, some care must be taken here, since we actually have no upper bound on $t$, and hence this statement is false as written; again, we ignore the issue in this overview.} from $H_{i+1}$; upon removing them $H_{i+1}$ is still $\Omega(d)$-connected, hence we can pick the path $P_{i+1}$ disjoint from the previous ones. At the end of this process, we can concatenate the obtained paths to produce a cycle of length at least $\Omega(td)\geq d$. Thus, we may assume that $G_4 = \varnothing$.

In case we are not already done, we have found that $G_2$ and $G_4$ are empty, and that $e(G_0)+e(G_3) \leq \eta dn$. Since we may take $\eta$ arbitrarily small, let us assume that $G_1$ has average degree $(1-o(1))d$. This means that some connected component $G_1'$ of $G_1$ also has average degree at least $(1-o(1))d$, as well as a vertex-cover of order at most $10d$. It is not hard to deduce from this that $G_1'$ has a subgraph $F$ with $O(d)$ vertices, which still has average degree at least $(1-o(1))d$; indeed, one can obtain $F$ by randomly sampling the vertices of $G_1'$ that are not in the cover, keeping each independently with probability $\Theta(d/\ab{V(G_1')})$. It now remains to prove that if $F$ is a graph with average degree $(1-o(1))d$ and $O(d)$ vertices, then $F$ contains a cycle of length $(1-c)d$. In other words, we have reduced the Erd\H os--Gallai theorem to a dense version of itself: rather than needing to prove it for all graphs, we only need to prove it for graphs whose order is a constant multiple larger than their average degree. While this is not obviously an easier task, it turns out to be quite simple using powerful tools such as Szemer\'edi's regularity lemma.

Rather than dwelling on the details, let us emphasize the most important point about the above sketch, which also recurs in the proofs of \cref{thm: main,thm:EG hyperstability}. This is that $G_2,G_3$ and $G_4$ are rather different from $G_1$. In each of $G_2,G_3$ and $G_4$, we have access to a structure that allows us to find a great deal more than we expect: even if the average degree is only $\eta d$, we are still able to find within it a cycle of length at least $4d$. Thus, the only case in which we do not find a cycle much longer than necessary is if essentially all the edges of $G$ are in $G_1$. That is, either we win by a lot, or we can essentially reduce to the case that $G$ is a disjoint union of pieces with small vertex-cover. 

This indicates how we go about proving \cref{thm:EG hyperstability}. Indeed, the sketch above already gives (essentially) the statement of \cref{thm:EG hyperstability}. To prove \cref{thm: main}, we need to make sure that each of the four arguments above can be made ``robust'', i.e.\ can be made to work even in the random subgraph $G_p$. For the ``large regular cores'' and the ``small vertex-cover'' cases, a standard application of Szemer\'edi's regularity lemma facilitates our arguments. The other two cases, on the other hand, and especially the ``well-connected pieces'' case, turn out to be more subtle, and we only explain these complications in the main body of the paper.

\vspace{4mm}
\par \noindent \textbf{Organization of the paper.} Section~\ref{sec:pokrovskiy} states Theorem~\ref{thm: structure}, which is a formal version of Theorem~\ref{thm: structure informal}, and contains a discussion of the necessity of the four cases. In the next six sections we assemble the proofs of Theorems~\ref{thm: main} and \ref{thm:EG hyperstability}.  Section~\ref{sec:preliminaries} introduces the basics of the DFS algorithm, Szemer\'edi's regularity lemma and other useful tools. Each of the next four sections is devoted to analyzing the four graphs $G_1,G_2,G_3$ and $G_4$ produced by Theorem~\ref{thm: structure}. Section~\ref{sec:final} brings everything together and proves Theorems~\ref{thm: main} and \ref{thm:EG hyperstability}. We conclude the paper with some final remarks in Section~\ref{sec: concluding}. The proof of Theorem~\ref{thm: structure} is located in Appendix~\ref{sec:appendix}.

Throughout the paper, we omit floor and ceiling signs whenever they are not crucial.

\section{Pokrovskiy's structure theorem}\label{sec:pokrovskiy}
\par In this section, we introduce and discuss a formal version of Pokrovskiy's structure theorem. Since the proof in full detail is quite involved and closely follows \cite{alexey}, we present the proof in the appendix. However, we give an (admittedly superficial) sketch of the proof in Section~\ref{sec:sketch}, so as to give some context for the various structures that appear in the statement. In Section~\ref{sec:statement}, we state Theorem~\ref{thm: structure}, the formal version of Pokrovskiy's structure theorem. Finally, in Section~\ref{sec:necessity}, we discuss a sense in which the structure theorem is tight, meaning each of $G_1,G_2,G_3$ and $G_4$ is necessary to include in the decomposition.

\subsection{A sketch of the proof of the structure theorem}\label{sec:sketch}

If a graph of average degree $d$ contains no ``dense spots'' (meaning $O(d)$ vertices spanning $\Omega(d^2)$ edges), then Theorem~\ref{thm: structure informal} is of course trivial to prove as we may simply set $G=G_3$. So the crux of Pokrovskiy's structure theorem is reasoning about the potential ways such dense spots may interact in a given graph. Given a dense spot, we may find in it a cut-dense subgraph via standard arguments (see Lemma~\ref{pok3.14}). Let us informally call such a cut-dense graph a \textit{blob}, and let $G$ be a graph edge-decomposed into blobs as $\bigcup_{i\in [k]} K_i$. A natural idea is to try to merge blobs together whenever they intersect on a large number of vertices. The fatal problem with this idea is that the union of two $\rho$-cut-dense graphs $K_1,K_2$ is not necessarily $\rho$-cut-dense, even if $K_1$ and $K_2$ intersect non-trivially. Thus, each merge potentially costs us some cut-density. Furthermore, there is no natural way to upper bound how many times a blob can be involved in a potential merging operation, so at the end of this process the original blobs may turn into subgraphs that are quite far from being cut-dense.
\par Pokrovskiy follows such a merging strategy, but keeping track of a different parameter. Given a blob $K_i$, let us approximate $K_i$ via Szemer\'edi's regularity lemma, and let us record by $M(K_i)$ the largest matching in the ``cluster graph'' of $K_i$ (this matching in the cluster graph may be converted into a dense and regular subgraph of the original graph $K_i$). Importantly, the vertex-set of this matching corresponds to a vertex-cover of the cluster graph, which in turn must cover almost all edges of the original graph $K_i$. Pokrovskiy's key insight is that (roughly speaking) it is $M(K_i)$ that one ought to attempt to track throughout the merging process above. Supposing $K_i$ and $K_j$ meet on, say, $\eps d$ vertices, we have one of the following two scenarios:
\begin{enumerate}
    \item Either $|M(K_i\cup K_j)|$ grows significantly (say by $\eps d$ elements) in comparison to $|M(K_i)|$ and $|M(K_j)|$, in which case we make progress towards building a structure with ``large regular cores'',
    \item or $|M(K_i\cup K_j)|$ doesn't grow much, in which case it must still have a small approximate vertex-cover, and hence the resulting structure is suitable for consideration for the ``small vertex-cover'' or the ``well-connected pieces'' parts of the structure theorem.
\end{enumerate}
The key point is that the first case can only happen a bounded number of times before a large regular core would emerge, giving us a way to still say something about the cut-density of the resulting structure.

\subsection{Formal statement of the structure theorem}\label{sec:statement}

We are finally in a position to state Pokrovskiy's structure theorem. A glance at the statement of the following theorem will suffice to recognize that the dependencies between different parameters are rather complex. While the more formal phrasing of Theorem~\ref{thm: structure} allows for comparison of parameters between different cases, we do not actually use this. For us, it is only important that we can freely choose $C,J,\varepsilon$ and $\delta$, that we can ensure $\zeta\ll\lambda$ and $\gamma\ll\zeta,\beta$, and that $d$ is large compared to all other constants.
\begin{theorem}\label{thm: structure}
    For every $C,J,\varepsilon,{\delta}>0$, there exist $\lambda,q,M>0$, such that for every $\zeta>0$, there exists $\beta>0$ such that for every $\gamma>0$ there exists $\rho>0$ such that for every sufficiently large $d$ and every graph $G$, there is an edge-decomposition of $G$ into $G_0,G_1,G_2,G_3,G_4$ satisfying the following properties.
    \begin{description}
        \item[Exceptional edges] The graph $G_0$ contains at most $\varepsilon e(G)$ edges.
        \item[Small vertex-cover] Every component of $G_1$ has a vertex-cover of size at most $Cd$. 
        \item[Large regular cores] $G_2$ is the union of edge-disjoint $\rho$-cut-dense graphs $H$, each of which contains a $qd$-regular subgraph of order at least $Cd$, and $\ab{V(H)}\leq Md$.  
        \item[Nowhere dense] $G_3$ does not contain a set of at most $Jd$ vertices inducing at least ${\delta} d^2$ edges.
        \item[Well-connected pieces] $G_4$ is the edge-disjoint union of a collection $\mathcal H$ of graphs satisfying the following properties.
        \begin{enumerate}
            \item Each $H\in\mathcal H$ contains a subgraph $Q$ of size at least $\beta d$ which is $\beta$-cut-dense and every vertex in $V(H)\setminus V(Q)$ has at least $\beta d$ neighbors in $V(Q)$.
            \item Any two $H,H'\in\mathcal H$ intersect in at most $\zeta d$ vertices.
            \item Every $H\in \mathcal H$ has at least $\lambda d$ vertices which intersect some other $H'\in\mathcal H$.
            \item For every $H\in\mathcal H$ there are at most $1/\lambda$ other $H'\in\mathcal H$ which intersect $H$ in more than $\gamma d$ vertices.
        \end{enumerate}
    \end{description}
\end{theorem}

\subsection{Necessity of the four graphs in Theorem~\ref{thm: structure informal}}\label{sec:necessity}
We now briefly discuss why each of the four types of graphs given by \cref{thm: structure informal} is unavoidable. That is, no version of this theorem is true if we omit one of these four types. Note that we address the informal Theorem~\ref{thm: structure informal} here instead of Theorem~\ref{thm: structure}, to avoid flooding the presentation with parameters.

First, for the necessity of $G_1$, take $G$ to be a disjoint union of copies of $K_{Cd/2}$. Suppose there exists a decomposition $G=G_0\cup G_2\cup G_3 \cup G_4$. It is easily seen that $G_0$ and $G_3$ each contain $o(dn)=o(e(G))$ edges. Furthermore, $G$ does not contain a connected subgraph of order at least $Cd$, so that $G_2$ must be empty. To conclude, $G_4$ must also be empty, for the following reason. Suppose not and let $H\in\mathcal H$. Let $\mathcal H'\subseteq \mathcal H$ be the set of $H'\in\mathcal H$ which intersect $H$. The properties of $\mathcal H$ guarantee that $|\mathcal H'|=\omega(1)$, each $H'\in \mathcal H$ contains at least $\Omega(d)$ vertices and the graphs in $\mathcal H'$ pairwise intersect in at most $o(d)$ vertices. This is enough to conclude that there are in fact $\omega(d)$ vertices appearing in $\bigcup_{H'\in\mathcal H'}V(H')$, which are contained in the connected component of $G$ containing $H$. But $G$ does not have a connected component of size $\omega(d)$, yielding a contradiction.

For the necessity of $G_2$, take $G$ to be the complete bipartite graph with parts $A$ and $B$ such that $|A|=2Cd$ and $|B|=n$, where we choose $n$ large compared to $d$. Suppose there exists a decomposition $G=G_0\cup G_1\cup G_3\cup G_4$. By definition, $G_0$ contains at most $o(dn)$ edges and the same goes for $G_3$, since if there are $\Omega(d)$ vertices in $B$ with degree $\Omega(d)$ in $G_3$ then it has a dense spot after removing all the other vertices of $B$. If $G_4$ is not empty then the collection of graphs $\mathcal H$ associated with $G_4$ necessarily contains $\omega(1)$ graphs. Each $H\in\mathcal H$ intersects $A$ in at least $\Omega(d)$ vertices since $Q\subseteq H$ is of order $\Omega(d)$ and $\Omega(1)$-cut-dense. But then, there must be $H,H'\in\mathcal H$ which intersect in $\Omega(d)$ vertices, a contradiction with the definition of $G_4$. Thus, $G_4$ is empty. Finally, $G_1$ is easily seen to contain at most $Cd\cdot n$ edges. As $G$ contains more than $o(dn)+Cd\cdot n$ edges, such a decomposition cannot exist. 

Next, for the necessity of the nowhere dense graph $G_3$. It is well-known that there exists a graph $G$ on $n$ vertices with average degree $d$ and girth $\omega(d)$, where $n$ is much larger than $d$. Then $G$ itself is already nowhere dense, and we now argue why it is impossible to partition $G$ as $G_0 \cup G_1 \cup G_2 \cup G_4$. By definition, $G_0$ contains only $o(e(G))$ edges and $G_2$ as well as $G_4$ must be empty as they contain a cycle of length $O(1)$ if they are not empty. If $G_1$ contains more than $n-1$ edges then it contains a cycle, but note that any cycle in $G_1$ has length at most $2Cd$, so $G_1$ contains at most $n-1$ edges. We conclude that $G_3$ is necessary since $G$ contains more than $o(e(G))+n-1$ edges.

Finally, let us discuss the necessity of $G_4$. The construction showing the necessity of $G_4$ is more complicated than in the other cases, so, for the sake of brevity, our argument here resembles more a sketch than a formal proof. Let $n$ be large compared to $Cd/2$, and begin by taking $F$ to be a random $Cd/2$-uniform hypergraph on $n$ vertices, where each $Cd/2$-set is included as an edge with probability $cn^{2-Cd/2}$ for some small enough $c>0$. By deleting a small number of edges, we may assume that $F$ is linear, that the Berge girth of $F$ is $\omega(1)$, and that every vertex of $F$ lies in $\Theta(n/(Cd/2-1)!)$ hyperedges. Now, for every hyperedge of $F$, form a graph $H$ by putting a copy of $K_{Cd/2}$ on that vertex set, and let $G$ be the union of the resulting graphs. Then all of $G$ can be viewed as $G_4$ (the linearity means that $\ab{H \cap H'}\leq 1$ for all $H,H'$, and the minimum degree of $F$ implies that every vertex of any $H$ intersects another $H'\in \mathcal H$).

Suppose towards a contradiction that there is a decomposition $G=G_0\cup G_1\cup G_2\cup G_3$. The number of edges in $G$ is $\Theta(n^2/(Cd/2-2)!)$, which in particular implies that $G$ is not dense as we consider $d$ to be very slowly growing in $n$. As in the previous cases, it is easily verified that $G_0$ and $G_3$ contain at most $o(n^{2}/(Cd/2-2)!)$ edges. Moreover, while each $H \in \mathcal H$ has a vertex-cover of size $Cd/2$, one cannot put a substantial portion of $G$ into $G_1$ without deleting a lot of edges, since every vertex lies in $\Theta(n/(Cd/2-1)!)$ graphs $H \in \mathcal H$, hence most of its degree is to other $H'\in\mathcal H$. Finally, $G_2$ is empty, since $G$ does not contain a subgraph whose size is between $Cd$ and $O(Cd)$ which is cut-dense. Indeed, every subgraph of $G$ of order at least $Cd$ either is not dense or has a cut-vertex. If one takes the union of more than a large constant number of $H \in \mathcal H$, the result will not be dense, but if one takes too few of them then by the Berge girth condition together with the linearity of $F$, there is a cut-vertex.

\section{Tools for the proofs of Theorems~\ref{thm: main} and \ref{thm:EG hyperstability}}\label{sec:preliminaries}

In this section, we introduce the key lemmas and definitions that we combine with Theorem~\ref{thm: structure} to prove Theorems~\ref{thm: main} and \ref{thm:EG hyperstability}.

\subsection{DFS-forests}
    In this subsection, we define a special type of forest, termed DFS-forest, as well as related terminology. Then, we explain in the next subsection how this concept relates to the depth-first search algorithm.
    
    Let $T$ be a rooted forest.  We denote by $<_T$ the poset induced by $T$, namely the poset whose elements are $V(T)$, and in which $u<_T v$ if and only if $u\neq v$ and $u$ is an ancestor of $v$ in $T$. In other words, two vertices are comparable if and only if they are on the same branch of the rooted forest and the ``smaller'' vertex is the one closer to the root.
    
    \par Posets induced by rooted forests have more structure than general posets; concretely, the following basic observation will be used repeatedly in the paper.
    \begin{observation}\label{key observation}
         If $u<_T w$ and $v<_T w$, then $u$ and $v$ are comparable.
    \end{observation}
    \begin{proof}
        Both $u$ and $v$ must be on the unique path between $w$ and the root of its component of $T$, hence they are comparable.
    \end{proof}
    The following definition is central to our proof. 

    \begin{definition}
    Let $G$ be a graph.
        We say that $T$ is a \emph{DFS-forest} of $G$ if $T$ is a rooted forest with $V(G)\subseteq V(T)$ such that for every edge $uv$ of $G$, $u$ and $v$ are comparable with respect to $<_T$.
    \end{definition}
     In the next section, we discuss the DFS-algorithm which finds a DFS-forest of any graph. We would like to stress here that we do not require $T\subseteq G$, and in fact we do not even require $V(G)=V(T)$.

    Given a rooted tree $T$ and $V\subseteq V(T)$, we write $T\llbracket V\rrbracket\subseteq T$ for the minimum subtree of $T$ containing $V$, and declare the root of $T\llbracket V\rrbracket$ to be the vertex closest to the root of $T$. We also write $T_v$ for the subtree of $T$ rooted at $v$ and $V_{T,v}$ for the descendants of $v$, that is $V_{T,v}=V(T_v)-v$. The \textit{$T$-distance} of two comparable vertices is defined to be the length of the path in $T$ connecting them. Similarly, the \textit{$T$-length} of an edge of $G$ is the $T$-distance between its two endpoints. We stress that whenever we refer to $T$-distance, it will only be between comparable vertices, and we do not consider the $T$-distance between incomparable vertices.

   We make the following observation: given a DFS-forest $T$ and an edge $e$ of $T$-length at least $d$, it follows that the unique cycle in $T\cup e$ has length at least $d$. This basic observation will drive our approach to the ``nowhere dense'' and ``well-connected pieces'' cases of the structure theorem; in order to find a long cycle in such graphs $G_3, G_4$, we will aim to find an edge of large $T$-length for an appropriate DFS-forest $T$.
\subsection{The depth-first search (DFS) algorithm}

    We use both the standard version of the depth-first search algorithm as well as a variant which combines it with percolation. The depth-first search algorithm, or the DFS-algorithm, is the following procedure: We are given a graph $G$. Throughout, we maintain a stack of vertices $S$, a set of processed vertices $V$ and a set of unvisited vertices $U$, maintaining the invariant that these three sets partition $V(G)$. We begin with $U=V(G)$ and $S,V=\varnothing$. Then, we iterate the following for as long as $S\cup U$ is not empty. If $S$ is empty, then select an arbitrary vertex in $U$ and move it to the top of $S$. In the case that $S$ is not empty, let $v$ be the vertex at the top of $S$. If $v$ has a neighbor in $U$, then pick an arbitrary such neighbor and move it to the top of $S$. Otherwise, we remove $v$ from $S$ and add it to the set of processed vertices $V$.

    We say that the DFS-algorithm uses an edge $uv$ if at some point $v$ is at the top of $S$ and we move $u$ from $U$ to $S$. The following fact is useful when we use the DFS-algorithm to find long paths.
    \begin{fact}\label{fact: DFS-path}
        The stack $S$ is a path in $G$ at every step of the DFS-algorithm.
    \end{fact}
    
    Next, let us introduce the percolated DFS-algorithm. Suppose we have some graph $G$ and a random subgraph $G_p$. The percolated DFS-algorithm proceeds by running the usual DFS-algorithm on $G$ but whenever we would like to use an edge $uv$ of $G$, we query its presence in $G_p$ and can only move $u$ from $U$ to $S$ if $uv$ is present in $G_p$. Similarly to the standard DFS-algorithm, we can say that an edge $uv$ was used if at some point $v$ was at the top of $S$ and $u$ got moved from $U$ to $S$. Note that every queried edge gets used if and only if it is present in $G_p$. Additionally, we say that an edge $uv\in G$ was queried by the percolated DFS-algorithm if, at some point, the presence of $uv$ in $G_p$ was queried.
    The following fact extends the well-known observation that the edges used in the standard DFS-algorithm form a DFS-forest.
    \begin{fact}\label{fact: DFS-forest}
        Let $T\subseteq G_p$ be the graph on $V(G)$ containing the edges which are used by the percolated DFS-algorithm. Then, $T$ is a DFS-forest of the graph $G'$ consisting of the edges that were not queried.
    \end{fact}
    \begin{proof}
        Observe, by ordering $V(G)$ in the order in which the vertices move from $U$ to $S$, that $T$ is $1$-degenerate. Thus, $T$ is a forest. Let us root each component of $T$ at the vertex which was the first to be moved to $S$. Consider an edge $uv\in E(G')$ and suppose that at some point, we move $u$ from $S$ to $V$ while $v\notin V$. Since $uv$ was not queried, it follows that $v\in S$. Then, the stack $S$, as it was right before moving $u$ to $V$, yields a path from the root of the component of $T$ containing $u$ to $u$ through $v$. It follows that $u$ and $v$ are comparable.
    \end{proof}
    A second crucial fact is that, while we cannot control how many edges the percolated DFS-algorithm queries, we do know that at most $n-1$ edges are used as $T$ is a forest.
    \begin{fact}\label{fact: DFS edges revealed}
        At most $n-1$ edges get used by the percolated DFS-algorithm.
    \end{fact}

\subsection{Szemer\'edi's regularity lemma and $m$-joinedness}
The arguments concerning $G_1$ and $G_2$ given by Theorem~\ref{thm: structure} build around the observation that both of these graphs can be treated like dense graphs. Arguably the most powerful tool applicable to dense graphs is Szemer\'edi's regularity lemma, which we introduce next.

For vertex sets $A,B$ in a graph $G$,
we define $d_G(A,B)=\frac{e_G(A,B)}{|A||B|}$.
\begin{definition}
    Given a graph $G$, a pair of disjoint subsets $A,B\subseteq V(G)$ is \textit{$\varepsilon$-regular} if for every $A'\subseteq A$ and $B'\subseteq B$ with $|A'|\geq\varepsilon|A|$ and $|B'|\geq\varepsilon|B|$ it holds that 
    $$
    |d_G(A,B)-d_G(A',B')|\leq\varepsilon.
    $$
\end{definition}
We state a well-known extension of Szemer\'edi's regularity lemma that refines an initially specified subset of vertices, and simultaneously regularizes two graphs on the same vertex set. An equipartition of a set $V$ into parts $V_0,\ldots, V_m$ is a partition of $V$ such that, for every $i$, $|V_i|\in\left\{\left\lfloor\frac{|V|}{m+1}\right\rfloor,\left\lceil\frac{|V|}{m+1}\right\rceil\right\}$.
\begin{lemma}[Szemer\'edi's regularity lemma]\label{lem: regularity}
    For every $\varepsilon,p>0$ and $L_0\in\mathbb N$ there exists $L\geq L_0$ such that the following holds. Let $G,H$ be two graphs on the same vertex set $V$ and let $U\subseteq V$ with $|U|\geq p|V|$. Then, there exists an equipartition $V=V_0\sqcup\ldots\sqcup V_m$ for some $L_0\leq m\leq L$ such that
    \begin{itemize} 
        \item for every $1\leq i\leq m$, $V_i$ is either contained in $U$ or disjoint from $U$, and 
        \item for every $1\leq i\leq m$ for all but at most $\varepsilon m$ choices of $1\leq j\leq m$ with $i\neq j$ we have that $(V_i,V_j)$ is $\varepsilon$-regular in $G$ and $H$.
    \end{itemize}
\end{lemma}
Percolated $\eps$-regular graphs maintain a notion of connectivity which we now introduce.

\begin{definition}
    A bipartite graph $G=(A,B)$ is called \textit{$m$-joined} if for every $A'\subseteq A$ and $B'\subseteq B$ with $|A'|=|B'|=m$, there is an edge in $G$ between $A'$ and $B'$.
\end{definition}

\begin{lemma}\label{lem: random joined}
    For every $\varepsilon>0$ the following holds for large enough $k$ and $t\geq k$. Let $(A,B)$ be an $\eps$-regular pair with $|A|,|B|=t$ of density at least $2\varepsilon$. For $p\geq k/t$, with probability at least $1-e^{-t}$, $(A,B)_p$ is $\varepsilon t$-joined.
\end{lemma}
\begin{proof}
    Let $A'\subseteq A,B'\subseteq B$ be of size exactly $\eps t$ each. By the regularity assumption, there are at least $\varepsilon\cdot \varepsilon^2t^2$ edges between $A'$ and $B'$. Therefore, the probability that at least one of them survives in $(A,B)_p$ is $1-(1-p)^{e(A',B')}\geq1-e^{-k\varepsilon^3t}$. Since there are at most $4^t$ choices of pairs $A'$ and $B'$, we get, choosing $k$ to be large enough, by a union bound that the statement holds with probability at least $1-e^{-t}$.
\end{proof}
It is well-known that $m$-joined graphs contain long paths. We give a proof for the sake of completeness.
\begin{lemma}\label{lem: DFS bipartite holes}
    Let $G=(A,B)$ be an $m$-joined, balanced bipartite graph. Then there exists a path in $G$ of length $|V(G)|-4m$.
\end{lemma}
\begin{proof} This follows by a standard argument using the (non-percolated) DFS-algorithm applied to $(A,B)$: Consider a time in the DFS-algorithm on $(A,B)$ where the size of $U$, the set of unvisited vertices, is exactly $2m$. Define $A'=A\cap U$ and $B'=B\cap U$. As the current stack $S$ is a path (see Fact~\ref{fact: DFS-path}), we may assume that it has size at most $|V(G)|-4m$. Therefore, there are at least $2m$ vertices which have been processed; we denote the processed vertices in $A$ by $A''$ and in $B$ by $B''$. Since there are no edges between processed and unvisited vertices, and by the $m$-joinedness assumption, one of $A',B''$ has size less than $m$ and also one of $B',A''$. Since there are exactly $2m$ vertices which are unvisited and also at least $2m$ vertices which are processed, it follows that without loss of generality, $|A'|,|A''|<m$. Using that $G$ is balanced bipartite, this implies that the stack contains at least $|A|-2(m-1)=(|V(G)|-4m)/2+2$ vertices of $A$ and at most $(|V(G)|-4m)/2-2$ vertices of $B$, which cannot be since the stack is a path. 
\end{proof}
We also give a slight variant of the above assertion.
\begin{lemma}\label{lem: regularity connector}
    Let $G$ be a graph on vertex set $U_1\cup\ldots\cup U_r$ such that for every $1\leq i< r$ the pair $(U_i,U_{i+1})$ is $m$-joined. Suppose further that each $U_i$ is of size at least $3m$ and let $Z\subseteq V(G)$ be of size at most $m$. Then there exists a path in $G$ from $U_1$ to $U_r$ avoiding $Z$ such that the $i$th vertex is in $U_i$ for all $i \in [r]$.
\end{lemma}
\begin{proof}
    For $1\leq i\leq r$, let $N_i\subseteq U_i\setminus Z$ be the vertices with neighbors in $N_{i-1}$, where $N_1=U_1\setminus Z$. We prove by induction on $i$ that $|N_i|\geq m$. The fact that $N_r$ is non-empty immediately implies the existence of the desired path. First, observe that $|N_1|\geq 2m$ so the base case holds. Next, suppose that the inductive hypothesis holds up to step $i-1\geq 1$, and we prove it for $i$. Indeed, the pair $(U_{i-1},U_i)$ is $m$-joined so there are at most $m$ vertices in $U_i$ which are not a neighbor of $N_{i-1}$. Thus, $|N_i|\geq |U_i|-|Z|-m\geq m$. 
\end{proof}

\subsection{Auxiliary results about the well-connected pieces case of Theorem~\ref{thm: structure}}
In Section~\ref{sec: well connected pieces} we will deal with graphs which have the structure of $G_4$ in Theorem~\ref{thm: structure}. To avoid cluttering that section, we record several preliminary results here. The \textit{intersection graph} of a hypergraph $\mathcal H$ is a graph on the edges of $\mathcal H$ where $e,e'\in\mathcal H$ are adjacent whenever $e$ and $e'$ intersect. Given a collection of graphs $\mathcal H$ as in the well-connected pieces case of Theorem~\ref{thm: structure}, one perspective from which we consider $\mathcal H$ is to look at the hypergraph on $V(G)$ where we place a hyperedge on $V(H)$ for every $H\in\mathcal H$. One important feature of this hypergraph is that its intersection graph robustly has high average degree in the sense of the following lemma.
\begin{lemma}\label{lem: average degree in cluster graph after removal}
    For every $\beta,\lambda,\gamma,\zeta>0$ with $\beta\leq \lambda/4$ and $\zeta\leq \lambda^2/8$ the following holds. Let $\mathcal H$ be a non-empty hypergraph such that the following hold.
    \begin{itemize}
        \item Any two $H,H'\in\mathcal H$ intersect in at most $\zeta d$ vertices.
        \item Every $H\in\mathcal H$ has at least $\lambda  d$ vertices which intersect some other $H'\in\mathcal H$;
        \item For every $H\in\mathcal H$ there are at most $1/\lambda $ other $H'\in\mathcal H$ which intersect $H$ in more than $\gamma d$ vertices.
    \end{itemize}
    Let $\mathcal H'$ be obtained by removing from every $H\in\mathcal H$ at most $\beta d$ vertices. Let $F$ be the intersection graph of $\mathcal H'$. Then $F$ contains at least $\lambda /(16\gamma)\cdot|\mathcal H|$ edges.
\end{lemma}
\begin{proof}
    Let $V\subseteq V(\mathcal H)$ be the set of vertices which appear in at least $2$ hyperedges. Let $G$ be the bipartite graph with vertex set $V\cup\mathcal H$ and we put an edge between $H\in\mathcal H$ and $v\in V$ if and only if $v\in H$. Note that by assumption every $H\in\mathcal H$ has degree at least $\lambda  d$ in $G$. Now, let $G'$ be obtained from $G$ by removing edges $vH$ if $v$ is removed from $H$ in $\mathcal H'$. Observe that $G'$ is obtained from $G$ by removing at most $\beta d$ edges incident to any $H\in\mathcal H$. Denote by $U\subseteq V$ the set of vertices which kept more than half of their incident edges and further define $G''=G'[U \cup \mathcal H]$. Thus we have removed at most $2\cdot \beta d\cdot |\mathcal H|$ many edges when passing from $G$ to $G''$. Note that $G$ has at least $\lambda  d\cdot |\mathcal H|$ edges so that $G''$ contains at least $\lambda  d|\mathcal H|/2$ edges since $\beta\leq\lambda/4$. Finally, we claim that if $H\in\mathcal H$ has degree $d_{G''}(H)\geq \lambda  d/4$ in $G''$, then $d_{F}(H)\geq d_{G''}(H)/(2\gamma d)$. We can then conclude using that
    $$
    2e(F)= \sum_{H\in\mathcal H'}d_F(H)\geq\sum_{H\in\mathcal H}d_{G''}(H)/(2\gamma d)\cdot\mathbf{1}_{d_{G''}(H)\geq \lambda  d/4}\geq (e(G'')-\lambda  d/4\cdot |\mathcal H|)/(2\gamma d)\geq \lambda /(8\gamma)\cdot |\mathcal H|.
    $$
    So, fix $H\in\mathcal H'$ with $d_{G''}(H)\geq\lambda  d/4$. As $\mathcal H'$ is obtained from $\mathcal H$ by removing vertices from hyperedges, it also holds that every $H'\in \mathcal H'$ intersects $H$ in at most $\zeta d$ vertices and there are at most $1/\lambda$ hyperedges $\mathcal H'$ which intersect $H$ in more than $\gamma d$ vertices. It follows, using $\zeta\leq\lambda^2/8$, that at least $d_{G''}(H)-\zeta d/\lambda \geq d_{G''}(H)/2$ vertices contained in $H$ are intersected by some $H'\in\mathcal H'$ which intersects $H$ in at most $\gamma d$ vertices. Therefore, there are at least $d_{G''}(H)/(2\gamma d)$ hyperedges $H'\in\mathcal H'$ which intersect $H$.
    \end{proof}
    The well-connected pieces case of Theorem~\ref{thm: structure} supplies more information than we need for our proof; the next lemma serves to extract the properties we need from this case.
    \begin{lemma}\label{lem: cut-dense properties}
    Let $H$ be a graph which contains a $\beta$-cut-dense subgraph $Q$ of size at least $\beta d$ and every vertex in $V(H)\setminus V(Q)$ has at least $\beta d$ neighbors in $V(Q)$. Then $H$ is $\beta^2 d/2$-connected and the diameter of $H$ is at most $2/\beta+2$.
\end{lemma}
\begin{proof}
We show that $Q$ is $\beta^2 d/2$-connected and has diameter at most $2/\beta$. The assertion then follows easily, since $H$ is obtained from $Q$ by adding vertices with at least $\beta d$ neighbors in $V(Q)$. This implies that the diameter of $H$ is at most two larger than that of $Q$ and that $H$ is also $\beta^2 d/2$-connected.

Set $m=|V(Q)|\geq \beta d$ and consider any tripartition $A\cup B\cup X$ of $ V(Q)$ with $|A|\leq |B|$ and $|X|<\beta^2 d/2$. By the cut-density of $Q$, there are at least $\beta|A||B\cup X|\geq \beta m|A|/2$ edges between $A$ and $B\cup X$ in $Q$. At most $|X||A|<\beta m|A|/2$ of those are incident to $X$ so that there is an edge between $A$ and $B$. It follows that $Q-X$ is connected, as claimed. For the second statement, let $v\in V(Q)$ be an arbitrary vertex. Denote by $N_i[v]$ the set of vertices at distance at most $i$ from $v$. Suppose that $|N_i[v]|\leq m/2$. Then, there are at least $\beta|N_i[v]|\cdot m/2$ edges leaving $N_i[v]$. Therefore, $N_{i+1}[v]$ contains at least $\beta m/2$ new vertices, i.e.\ vertices not present in $N_i[v]$. It follows that $|N_{1/\beta}[v]|> m/2$. As a consequence, the diameter of $Q$ is at most $2/\beta$, since $N_{1/\beta}[v]$ and $N_{1/\beta}[u]$ intersect for every $u,v\in V(Q)$.
\end{proof}
    We will also use the following two elementary and well-known graph-theoretic statements.
    \begin{lemma}\label{lem: get path from average degree}
    Let $F$ be a graph of average degree $k$. Then there exists $v\in V(F)$ and a path $P\subseteq F$ such that $v$ has at least $k/2$ neighbors in $P$.
\end{lemma}
\begin{proof}
    By passing to a subgraph, we may assume that $F$ has minimum degree at least $k/2$. Now letting $P$ be any longest path, and $v$ one of its endpoints, gives the desired statement.
\end{proof}
\begin{lemma}\label{lem: independent set}
    Every graph $G$ contains an independent set of size at least $|V(G)|/(\Delta(G)+1)$.
\end{lemma}
\begin{proof}
    By coloring greedily, we obtain a proper coloring of $G$ with $\Delta(G)+1$ colors, and hence the largest color class is an independent set of at least the desired size.
\end{proof}

\section{Small vertex-cover}
We start by showing graphs with small vertex-covers contain a dense subgraph with essentially the same average degree. The same argument is used in \cite{alexey}; we reproduce it below for the sake of completeness.
\begin{lemma}\label{lem: cover to dense}
    For every $C,\tau> 0$ there exists $\ell\in\mathbb{N}$ such that the following holds for large enough $d$. Let $G$ be a graph of average degree at least $d$ and with a vertex-cover of size at most $Cd$. Then there exists $G'\subseteq G$ with at most $\ell d$ vertices and average degree at least $(1-\tau)d$.
\end{lemma}
\begin{proof}
    Let $\ell'$ be large enough compared to $C$ and $1/\tau$ and we set $\ell=\ell'+C$. Let $V\subseteq V(G)$ be a minimum vertex-cover of $G$. By assumption, we have $|V|\leq Cd$. Set $N=|V(G)\setminus V|$. If $N<\ell'd$ then $|V(G)|\leq (\ell' +C)d= \ell d$, so the statement trivially follows for $G'=G$. Therefore, let us assume $N\geq \ell' d$. Let $G'=G[V\cup U]$, where $U$ is obtained by selecting uniformly at random $\ell' d$ vertices of $V(G)\setminus V$. Note that $G'$ is guaranteed to have at most $(C+\ell')d= \ell d$ vertices so it suffices to show the expected average degree of $G'$ is at least 
    $(1-\tau)d$ and apply the probabilistic method. Observe that every edge of $G$ is in $G'$ with probability at least $\ell' d/N$. Note that $e(G)\geq dn/2\geq dN/2$ and $|V(G')|\leq (C+\ell') d\leq \ell' d/(1-\tau)$, where we use that $\ell'$ is large compared to $C$ and $1/\tau$. It follows that the expectation of $e(G')$ is at least $\ell' d/N\cdot dN/2=\ell' d^2/2$. Therefore, the expected average degree of $G'$ is at least $\ell' d^2/|V(G')|\geq (1-\tau)d$.
\end{proof}
Below we show a version of our main theorem (Theorem~\ref{thm: main}) in the special case of dense graphs. Combined with the above lemma, this result handles the ``small vertex-cover'' case of the structure theorem.
\begin{lemma}\label{lem: dense case}
    For every $\tau,\ell>0$ the following holds for all large enough $K$ and every $d\geq K$ as well as $p\geq K/d$. Let $G$ be a graph of average degree at least $d$ on at most $\ell d$ vertices. Then $G_p$ contains a cycle of length at least $(1-\tau)d$ with probability at least $1-e^{-\Omega(d)}$.
\end{lemma}
\begin{proof}
Fix parameters such that $\tau, \ell^{-1}\gg \varepsilon\gg L_0^{-1}\gg K^{-1}\geq d^{-1}$.
Let $V(G)=V_0\sqcup\ldots\sqcup V_{m}$ as in Lemma~\ref{lem: regularity} applied with $H$ as an empty graph, $U=V(G)$ and $p=1$. For simplicity of presentation, let us assume that each $V_i$ has size exactly $t=|V(G)|/(m+1)$. By Lemma~\ref{lem: random joined} and the union bound, we have that, with probability at least $1-m^2e^{-t}=1-e^{-\Omega(d)}$, for every $1\leq i<j\leq m$ with $(V_i,V_j)$ $\varepsilon$-regular of density at least $2\varepsilon$ in $G$, $(V_i,V_j)$ is $\varepsilon t$-joined in $G_p$. We show that whenever this is the case then $G_p$ contains a cycle of length at least $(1-\tau)d$.

So, let us assume that $(V_i,V_j)$ is $\varepsilon t$-joined in $G_p$ for every $1\leq i<j\leq m$ with $(V_i,V_j)$ $\varepsilon$-regular of density at least $2\varepsilon$.
Let $\mathcal R$ be the auxiliary graph on $\{V_1,\ldots,V_m\}$ where we join $V_i$ and $V_j$ whenever $(V_i,V_j)$ is $\varepsilon$-regular with density at least $2\varepsilon$ in $G$. Each edge in $\mathcal R$ corresponds to at most $t^2$ edges of $G$. For every $1\leq i<j\leq m$ with $(V_i,V_j)$ not $\varepsilon$-regular there are at most $t^2$ edges and if $(V_i,V_j)$ is of density at most $2\varepsilon$ then there are at most $2\varepsilon t^2$ edges. Furthermore, each $V_i$ contains at most $\binom{t}{2}$ edges and $V_0$ is incident to at most $t|V(G)|=(m+1)t^2$ edges in $G$. It follows that 
$$
e(\mathcal R)\geq \frac{e(G)-m\cdot\varepsilon m\cdot t^2-\binom{m}{2}\cdot 2\varepsilon t^2-m\cdot \binom{t}{2}-(m+1)t^2}{t^2}\geq d/t\cdot m/2-2\varepsilon m^2-3m. 
$$
Using that $d\geq |V(G)|/\ell$, we get that $d/t\geq m/\ell$. So $e(\mathcal R)\geq (1-4\varepsilon\ell-6\ell/m)d/t\cdot m/2\geq (1-\tau/2)d/t\cdot m/2$.

By the Erd\H{o}s--Gallai theorem, $\mathcal R$ contains a cycle $C$ of length at least $(1-\tau/2)d/t$. Let us relabel the sets $V_1,\ldots,V_m$ such that $C=V_1,\ldots, V_{|C|}$. By Lemma~\ref{lem: DFS bipartite holes}, for every even $1\leq i\leq |C|$ there is a path $P_i\subseteq G_p$ between $V_i$ and $V_{i+1}$ of length at least $(2-4\varepsilon)t$. Arbitrarily direct each $P_i$ and let $I_i$ be the initial segment of length $6\varepsilon t$, $O_{i+1}$ the last segment of length $6\varepsilon t$ and $P_i'$ the middle segment of length $(2-16\varepsilon)t$. Then $I_i$ intersects $V_i$ in $3\varepsilon t$ vertices and $O_{i+1}$ intersects $V_{i+1}$ in $3\varepsilon t$ vertices. If $|C|$ is even we can now connect up these paths to a cycle containing $P_i'$ for every even $1\leq i\leq |C|$ since there is an edge between $O_{i+1}$ and $I_{i+2}$ in $G_p$ for every even $1\leq i\leq |C|$. If $|C|$ is odd, note that by Lemma~\ref{lem: regularity connector} (with $Z=\varnothing$) there exists a path in $G_p$ of length $2$ from $O_{|C|}\cap V_{|C|}$ to $I_2\cap V_2$ through $V_1$ since both $(V_{|C|},V_1)$ and $(V_{1},V_2)$ are $\varepsilon t$-joined in $G_p$. Therefore, we obtain a cycle in $G_p$ of length at least 
\[
\sum_{\text{even }i=1}^{|C|} |P_i'|\geq 
(|C|-1)/2\cdot (2-16\varepsilon)t\geq ((1-\tau/2)d/t-1)\cdot (1-8\varepsilon)t\geq (1-\tau)d.\qedhere
\]
\end{proof}
\section{Large regular cores}
This section proves the following lemma designed to address the case when $G_2$ from the structure theorem is not empty.

\begin{lemma}\label{lem:cut-dense case} For every $C,M,\rho,q,\xi>0$ the following holds for large enough $K$ and every $d\geq K$. Let $H$ be a $\rho$-cut-dense graph of order at most $Md$ which contains a $qd$-regular subgraph of order at least $Cd$. For $p\geq K/d$, $H_p$ contains a cycle of length $(1-\xi)Cd$ with probability at least $1-e^{-\Omega(d)}$.
\end{lemma}

In its proof we will use the following well-known extension of Vizing's theorem (proved independently by Gupta) to multigraphs; see e.g.\ \cite[Theorem 3.4.4]{MR5044348}.

\begin{theorem}[Vizing--Gupta theorem]
    Let $G$ be a multigraph of multiplicity at most $\mu$. Then, the edge-chromatic number of $G$ is at most $\Delta(G)+\mu$.
\end{theorem}
We may now begin the proof.
\begin{proof}[Proof of Lemma~\ref{lem:cut-dense case}]
We fix constants satisfying the hierarchy: $1/C,1/M,\rho,q,\xi \gg \varepsilon \gg L_0^{-1} \gg K^{-1} \gg d^{-1}$. Let $G\subseteq H$ be a $qd$-regular subgraph of order at least $Cd$.

Let $V_0,\ldots,V_m$ be obtained from Lemma~\ref{lem: regularity} applied to $G$ and $H$ with $\varepsilon, p=C/M\leq |V(G)|/|V(H)|, L_0$ and $U=V(G)$. By relabeling if need be, let us assume that $V(G)\setminus V_0=V_1\sqcup\ldots\sqcup V_{m_G}$. For simplicity of presentation, we assume henceforth that each $V_i$ has size exactly $t=|V(H)|/(m+1)$. For every $1\leq i\leq m$, fix an arbitrary subset $V_i'\subseteq V_i$ of size $3\varepsilon t$.

The following is the only use we make of randomness. By Lemma~\ref{lem: random joined} and a union bound, with probability at least $1-2m^2e^{-t}=1-e^{-\Omega(d)}$ for each $1\leq i<j\leq m$, if $(V_i,V_j)$ is $\varepsilon$-regular with density at least $2\varepsilon$ in either $H$ or $G$, then it is $\varepsilon t$-joined in $H_p$. For the remaining argument, let us fix an outcome of the randomness where this holds.

For each $1\leq i<j\leq m_G$, set $\mu(V_i,V_j)=\left\lfloor\frac{e_G(V_i,V_j)}{2\varepsilon t^2}\right\rfloor$ whenever $(V_i,V_j)$ is $\varepsilon$-regular in $G$, otherwise set $\mu(V_i,V_j)=0$. Let $\mathcal R_G$ be the multigraph on $\{V_1,\ldots,V_{m_G}\}$ where we add exactly $\mu(V_i,V_j)$ edges between every pair $V_i,V_j$ with $1\leq i<j\leq m_G$. Next, we aim to show that $\mathcal R_G$ is almost regular so that we can apply the Vizing--Gupta theorem to conclude that $\mathcal R_G$ contains an almost spanning matching. For $1\leq i\leq m_G$, $V_i$ is incident to at most $t\cdot qd$ edges in $G$. Note also that $V_i$ is incident to at least $\deg_{\mathcal{R}_G}(V_i)\cdot 2\varepsilon t^2$ edges in $G$. We get that $\deg_{\mathcal{R}_G}(V_i)\leq qd/(2\varepsilon t)$. As this holds for all $1\leq i\leq m_G$, it is also an upper bound on $\Delta(\mathcal{R}_G)$. On the other hand, for every $1\leq i<j\leq m_G$, we have that $e_G(V_i,V_j)\leq(\mu(V_i,V_j)+1)\cdot 2\varepsilon t^2$ whenever $(V_i,V_j)$ is $\varepsilon$-regular in $G$ and otherwise $e_G(V_i,V_j)\leq t^2$. Furthermore, for every $1\leq i\leq m_G$, we have $e_G(V_i)\leq \binom{t}{2}$ and $e_G(V_0,V(G)\setminus V_0)\leq t\cdot qd$.

Therefore, we get
\begin{align*}
    qd|V(G)|/2=e(G)\leq e(\mathcal{R}_G)\cdot 2\varepsilon t^2+\binom{m_G}{2}\cdot 2\varepsilon t^2+m_G\cdot\varepsilon m\cdot t^2+m_G\cdot \binom{t}{2}+t\cdot qd.
\end{align*}
Recall that $|V(G)|\geq m_G\cdot t$ and that $(m+1)t=|V(H)| \leq Md$, and hence $d/t\geq m/M$. Dividing by $2\varepsilon t^2$ and rearranging now yields
\begin{align*}
    e(\mathcal R_G) \geq qd/(4\varepsilon t)\cdot m_G- \binom{m_G}{2}-m_Gm/2-m_G/(4\varepsilon)-qd/(2\varepsilon t)\geq (1-5M\varepsilon/q-2/m_G)qd/(4\varepsilon t)\cdot m_G. 
\end{align*}
Using that $1/m_G\ll M\varepsilon/q$, we get 
$$
e(\mathcal R_G) \geq (1-6M\varepsilon/q)qd/(4\varepsilon t)\cdot m_G. 
$$
Finally, denote by $\mu(\mathcal R_G)$ the maximum edge multiplicity of $\mathcal R_G$. Then $\mu(\mathcal R_G)\leq 1/(2\varepsilon)$. By the Vizing--Gupta theorem, we have that $\mathcal R_G$ has a proper edge-coloring with $\Delta(\mathcal R_G)+\mu(\mathcal R_G)\leq (qd+t)/(2\varepsilon t)$ colors. By averaging, one of those colors contains at least
$$
\frac{e(\mathcal R_G)}{(qd+t)/(2\varepsilon t)}\geq (1-6 M \varepsilon /q)m_G/2\cdot\frac{qd}{qd+t}\geq (1-7M\varepsilon/q)m_G/2
$$
edges. Let $F\subseteq \mathcal R_G$ be the matching corresponding to this color and note that $|V(F)|\geq (1-7M\varepsilon/q) m_G$.

For $V_iV_j\in E(F)$, we have that $(V_i\setminus V_i',V_j\setminus V_j')$ is $\varepsilon t$-joined in $H_p$, since $(V_i,V_j)$ is $\varepsilon t$-joined by assumption and, hence, every induced subgraph of $(V_i,V_j)$ is $\varepsilon t$-joined. By Lemma~\ref{lem: DFS bipartite holes}, there is a path of length $|V_i\setminus V_i'|+|V_j\setminus V_j'|-4\varepsilon t = 2t-10\varepsilon t$ in $H_p$.  Let $P_1,\ldots,P_{|F|}$ be such paths for each edge in $F$. We aim to connect these paths into a cycle. For every $P_i$, give it an arbitrary direction and let $I_i$ be the initial segment of $6\varepsilon t$ vertices, $O_i$ the segment of the last $6\varepsilon t$ vertices and $P_i'$ the middle segment which is of length at least $2t-22\varepsilon t$. For every $1\leq i\leq |F|$, let $W_i$ denote one of the two clusters which intersect $P_i$.

Let $\mathcal R_H$ be the graph on $\{V_1,\ldots, V_m\}$ where $V_i,V_j$ are connected if $(V_i,V_j)$ is $\varepsilon$-regular of density at least $2\varepsilon$ in $H$.
    \begin{claim}\label{claim: cluster graph is connected}
        $\mathcal R_H$ is connected.
    \end{claim}
    \begin{proof}
        Suppose otherwise, without loss of generality, let us say that $\mathcal R_H$ does not contain an edge between $\{V_1,\ldots, V_\ell\}$ ($\ell\leq m/2$) and $\{V_{\ell+1},\ldots,V_m\}$. Then, all edges in $H$ between $A=\bigcup_{i=1}^\ell V_i$ and $V(H)\setminus A$ are between pairs $(V_i,V_j)$ with either $1\leq i\leq\ell$ and $j=0$, or, $1\leq i\leq\ell<j\leq m$ and either $(V_i,V_j)$ is not $\varepsilon$-regular or $(V_i,V_j)$ has density less than $2\varepsilon$. The number of such edges is at most
        $$
        \ell t^2+\ell\varepsilon m\cdot t^2+\ell(m-\ell)\cdot2\varepsilon t^2\leq \ell(m+1-\ell)t^2\cdot 4\varepsilon< \rho\ell(m+1-\ell)t^2,
        $$
        using that $\varepsilon \ll \rho$.
        Noting that $\ell t=|A|$ and $(m+1-\ell)t=|V(H)\setminus A|$ this contradicts the assumption that $H$ is $\rho$-cut-dense.
    \end{proof}
    We iteratively find disjoint paths $Q_i$ in $H_p$ for every $1\leq i\leq |F|$ of length at most $m$ each such that $Q_i$ connects $O_i$ with $I_{i+1}$ (indices mod $|F|$) and the internal vertices are in $\bigcup_{j=1}^m V_{j}'$. Suppose we have already found $Q_1,\ldots, Q_{i-1}$. Set $Z=\bigcup_{j=1}^{i-1}V(Q_{j})$ and note that $|Z|\leq m^2\leq \varepsilon t$, using that $d^{-1}\ll \varepsilon$. Let $P=(V_{j_1},\ldots, V_{j_r})$ be a path in $\mathcal R_H$ with $V_{j_1}=W_i$ and $V_{j_r}=W_{i+1}$. Since $\mathcal R_H$ is a connected $m$-vertex graph, such a path exists and is of length at most $m$. Since $V_{j_h}V_{j_{h+1}}\in \mathcal R_H$, the pair $(V_{j_h},V_{j_{h+1}})$ is $\varepsilon t$-joined in $H_p$. Note that $|O_i\cap W_i|\geq 3\varepsilon t$ and also $|I_{i+1}\cap W_{i+1}|\geq 3\varepsilon t$. By Lemma~\ref{lem: regularity connector} applied with $U_1=O_i\cap W_i$, $U_2 = V_{j_2}',\ldots, U_{r-1}=V_{j_{r-1}}'$ and $U_r=I_{i+1}\cap W_{i+1}$, there exists a desired path $Q_i$ in $H_p$. 

    Finally, concatenating all the pieces, we conclude that there is a cycle $C$ in $H_p$ which includes $P_i'$ for all $1\leq i\leq |F|$.
    Thus, the total length of this cycle is at least 
    $$\sum_{i=1}^{|F|}|P_i'|\geq (1-7M\varepsilon/q)m_G/2\cdot (2-22\varepsilon)t\geq (1-7M\varepsilon/q)(1-12\varepsilon)|V(G)|\geq (1-\xi)Cd,$$
    using that $\varepsilon \ll \xi,q,M^{-1}$. This concludes the proof.
\end{proof}

\section{Nowhere dense}
If $G_3$ in the structure theorem retains any positive fraction of the total edge-set, the following lemma will allow us to find long cycles in (percolated subgraphs of) $G$ by showing the existence of long edges with respect to a given DFS-forest $T$.
\begin{lemma}\label{lem: nowhere dense case}
    For every $\eta>0$ the following holds for small enough $\delta>0$. Let $G$ be a graph with average degree at least $\eta d$ which does not contain a set of at most $(1+\eta/8)d$ vertices inducing at least $ \delta d^2$ edges. Let $T$ be a DFS-forest of $G$.
    Then $G$ contains an edge of $T$-length at least $d$.
\end{lemma}
\begin{proof}
    Without loss of generality, let us assume that $G$ is connected and has minimum degree at least $\eta d/2$, where we recall that the definition of a DFS-forest implies that $T$ is a DFS-forest of any subgraph of $G$. Let $D_{\mathrm{start}}$ be an orientation of $G$ where every edge is directed away from the root of $T$. Given a digraph $D$ on $V(G)$, we say that $v\in V(G)$ is $(D,T)$-terminal if $v$ is the only vertex in $V(T_v)$ which has an out-edge in $D$. Note that, as long as $D$ has at least one arc, there exists a $(D,T)$-terminal vertex. Indeed, let $(v,w)\in D$ be chosen such that $v$ is as far away in $T$-distance from the root of $T$ as possible. Therefore, no vertex in $V_{T,v}$ has an out-edge in $D$. Thus, $v$ is $(D,T)$-terminal. 

    Consider the following procedure initialized with $D=D_{\mathrm{start}}$. As long as there exists a $(D,T)$-terminal vertex $v$ with $d^+_D(v)\leq \eta d/8$, we update $D$ by removing all out-edges from $v$, where $d^+_D(v)$ denotes the out-degree of $v$ in $D$.
    
    Note that every digraph $D$ encountered by the algorithm is a subgraph of $D_{\mathrm{start}}$ and that we remove at most $ \eta d/8$ arcs in a single step. Let $D_{\mathrm{end}}$ be the digraph at the end of this algorithm and note that it terminates as each vertex is considered at most once.

    Observe that we remove at most $\eta d/8\cdot |V(G)|$ arcs in the procedure, so that $D_{\mathrm{end}}$ contains an arc and therefore, also a $(D_{\mathrm{end}},T)$-terminal vertex $v$. By the definition of the procedure, we have $d^+_{D_{\mathrm{end}}}(v)>\eta d/8$. Since all the out-neighbors of $v$ in $D_{\mathrm{end}}\subseteq D_{\mathrm{start}}$ lie in $V_{T,v\cap V(G)}$, it follows that
    $|V_{T,v}\cap V(G)|>\eta d/8$. Let $P_v\subseteq T$ be the path in $T$ from the root to $v$. As every vertex in $V_{T,v}$ has no out-edges, by the definition of a DFS-forest, we get that all the arcs of $D_{\mathrm{end}}$ with one endpoint in $V_{T,v}$ have the other endpoint in $V(P_v)$. 
    \par We now claim that
    \begin{equation}\label{eq:terminal claim}
    \sum_{u\in V_{T,v}\cap V(G)}d_{D_{\mathrm{end}}}^-(u)\geq \frac{\eta d}8\cdot |V_{T,v}\cap V(G)|,    
    \end{equation}
    where $d^{-}_{D_{\mathrm{end}}}(u)$ denotes the in-degree of $u$ in $D_{\mathrm{end}}$.
    Before proving \eqref{eq:terminal claim}, 
    let us first show how it can be used to conclude the proof. First, recall that $D_{\mathrm{end}}\subseteq D_{\mathrm{start}}$. Therefore, every arc in $D_{\mathrm{end}}$ corresponds to an edge in $G$. Assuming \eqref{eq:terminal claim}, by averaging, we can find some subset $V\subseteq V_{T,v}\cap V(G)$ of size $\eta d/8$ incident to at least $\eta^2 d^2/64$ arcs in $D_{\mathrm{end}}$. Define $W\subseteq V(P_v)$ to be the $d$ vertices on $P_v$ closest to $v$ (define $W$ to be $V(P_v)$ if there are fewer than $d$ vertices on $P_v$). Note that any edge between $V_{T,v}$ and $V(P_v)\setminus W$ is of $T$-length at least $d$. But $|V\cup W|\leq (1+\eta/8)d$, so $G[V\cup W]$ can contain at most $\delta d^2<\eta^2 d^2/64$ edges, implying the existence of an edge of $T$-length at least $d$. 
    \par We now verify that \eqref{eq:terminal claim} holds. For every digraph $D$ encountered by the procedure and $w\in V(D)$, denote by $W_{w,D}\subseteq V_{T,w}\cap V(G)$ the set of vertices with no out-edges in $D$. We show by induction that in every step of the procedure it is true that for every vertex $w\in V(D)$ with at least one out-edge in $D$ it holds that $\sum_{u\in W_{w,D}}d_{D}^-(u)\geq \eta d|W_{w,D}|/8$. Observe that when applied to the final digraph $D_{\mathrm{end}}$, for a $(D_{\mathrm{end}},T)$-terminal vertex $v$ this is equivalent to \eqref{eq:terminal claim}, as $W_{v,D_{\mathrm{end}}}=V_{T,v}\cap V(G)$.
    
    Before the first step, the inequality holds trivially since the minimum degree of $G$ is at least $\eta d/2$. Let $D_0$ be the digraph at some point in the procedure and let $D_1$ denote the digraph obtained by the procedure from $D_0$ by removing the out-edges from $v$. Suppose that for every vertex $w\in V(D_0)$ with at least one out-edge in $D_0$ it holds that $\sum_{u\in W_{w,D_0}}d_{D_0}^-(u)\geq \eta d|W_{w,D_0}|/8$. We show that then $D_1$ also has this property. Let $w$ be an arbitrary vertex which has an out-edge in $D_1$.
    As $v$ is $(D_0,T)$-terminal and $w$ has an out-edge in $D_1\subseteq D_0$, we get that $w\notin V_{T,v}$. Therefore, if $v\notin V_{T,w}$ then $$
    \sum_{u\in W_{w,D_1}}d_{D_1}^-(u)=\sum_{u\in W_{w,D_0}}d_{D_0}^-(u)\geq \eta d|W_{w,D_0}|/8,
    $$
    so that we are done. So, suppose that $v\in V_{T,w}$. By definition, we have $W_{w,D_1}\setminus W_{w,D_0}=\{v\}$. Also, note that the arcs incident to $v$ in $D_0$ did not change from those in the beginning of the procedure so that $d^-_{D_0}(v)+d^+_{D_0}(v)\geq \eta d/2$ and since, by the definition of the procedure, $d^+_{D_0}(v)\leq \eta d/8$, $d^-_{D_0}(v)\geq \eta d/4$ with room to spare.
    We get that
    $$
        \sum_{u\in W_{w,D_1}}d_{D_1}^-(u)= \sum_{u\in W_{w,D_0}}d_{D_0}^-(u) + d_{D_0}^-(v)-d_{D_0}^+(v)\geq\eta d|W_{w,D_0}|/8+\eta d/8 = \eta d|W_{w,D_1}|/8,
    $$
thereby concluding the proof.
\end{proof}

\section{Well-connected pieces}\label{sec: well connected pieces}
In this section we prove the following result, which is designed to address the well-connected pieces case of the structure theorem. This is arguably the most delicate part of our argument.
\begin{lemma}\label{lem:main cut-dense union}
    For every $\lambda>0$ there exists $\zeta>0$ such that for every $\beta>0$ there exists $\gamma>0$ such that for sufficiently large $d$ the following holds.
    Let $G$ be a non-empty graph which is the edge-disjoint union of a collection $\mathcal H$ of graphs satisfying the following properties.
    \begin{enumerate}[ref=(\arabic*)]
        \item Each $H\in\mathcal H$ contains a subgraph $Q$ of size at least $\beta d$ which is $\beta$-cut-dense and every vertex in $V(H)\setminus V(Q)$ has at least $\beta d$ neighbors in $V(Q)$.\label{hyp:beta}
        \item Any two $H,H'\in\mathcal H$ intersect in at most $\zeta d$ vertices.\label{hyp:zeta}
        \item Every $H\in \mathcal H$ has at least $\lambda d$ vertices which intersect some other $H'\in\mathcal H$.\label{hyp:lambda}
        \item For every $H\in\mathcal H$ there are at most $1/\lambda$ other $H'\in\mathcal H$ which intersect $H$ in more than $\gamma d$ vertices.\label{hyp:gamma}
    \end{enumerate}
    Furthermore, let $T$ be a DFS-forest of $G$. Then $G$ contains an edge of $T$-length at least $d$.
\end{lemma}
For the rest of the section we operate under the assumptions of Lemma~\ref{lem:main cut-dense union}. We choose $\zeta>0$ to be small enough compared to $\lambda$. Without loss of generality, we may assume that $\beta>0$ is small enough compared to $\lambda,\zeta$ (as this only weakens the hypotheses on $\mathcal{H}$) and we choose $\gamma>0$ to be small enough compared to $\lambda,\zeta,\beta$. We set $\beta'=\beta^2/2$. We obtain the hierarchy $\lambda\gg\zeta\gg\beta'\gg\gamma\gg d^{-1}>0$.  

We remark here, that \ref{hyp:beta} is stronger than what we need here. In fact, we will only use that every $H\in\mathcal H$ is $\beta' d$-connected and has diameter at most $1/\beta'$, both guaranteed by Lemma~\ref{lem: cut-dense properties}.

Recall that we have a fixed DFS-forest $T$ of $G$, and that $G$ is the edge-disjoint union of the graphs in $\mathcal{H}$. We show that every member $H\in \mathcal{H}$ contains a linear-sized segment that is ``small'' with respect to the partial order defined by $T$. These segments will serve as the building blocks as we search for an edge of $G$ that is long with respect to $T$.

\begin{lemma}\label{lem: cluster in dfs tree}
    For every $H\in\mathcal H$ there is $V\subseteq V(H)$ with $|V|=\beta' d$, such that, for all $v\in V$ and $u\in V(H)$, $v$ and $u$ are comparable with respect to $<_T$. Additionally, if $v\in V$ and $u\in V(H)\setminus V$ then $v<_Tu$.
    
    \end{lemma}
\begin{proof}
    First, we prove the weaker statement that for any connected subgraph $G'\subseteq G$ there is $v\in V(G')$ such that $v\leq_T u$ for all $u\in V(G')$ (or in other words $v$ is comparable to and smaller than all other vertices of $V(G')$). Let $v$ be the root of $T\llbracket{V(G')}\rrbracket$. Then, $v\leq_T u$ for every $u\in V(G')$, so it is left to show that $v\in V(G')$. Suppose not. Then, there is a child $w$ of $v$ in $T$ such that $S=V(T_w)\cap V(G')$ is neither empty nor $V(G')$. Let $st$ be an edge of $G'$ with $s\in S$ and $t\notin S$, which exists as $G'$ is connected. Since $T$ is a DFS-forest, $s$ and $t$ must be comparable, so $t<_T w$, as $t$ is not in $T_w$. But $t$ is also not $v$, as $v\notin V(G')$, so we get that $t\notin V(T_v)$, contradicting that $T\llbracket{V(G')}\rrbracket\subseteq T_v$.

    Using the statement we just proved, we can now finish the proof. Set $H_1=H$. Since $H_1$ is connected, there is $v_1\in V(H_1)$ comparable to and smaller than all other vertices of $H_1$. Set $H_2=H_1-v_1$. Since $H$ is $\beta' d$-connected, $H_2$ is connected. Similarly as before, we find $v_2$ and remove it to obtain $H_3=H_2-v_2$. Continuing this process $\beta'  d$ times, we obtain the desired set of vertices.
\end{proof}
Lemma~\ref{lem: cluster in dfs tree} allows us to introduce notation describing the interaction of the graphs in $\mathcal H$ with $T$. For every $H\in\mathcal H$, let $V_H\subseteq V(H)$ be the set given by Lemma~\ref{lem: cluster in dfs tree} of size $\beta' d$. Lemma~\ref{lem: cluster in dfs tree} guarantees that every pair of vertices in $V_H$ is comparable, so that we can define $f_H$ to be the first and $w_H$ the last vertex of $V_H$. Note that it is not necessarily the case that $V_H$ induces a subpath of $T$ (in fact, $T$ may even have vertices not included in $G$); rather, all we know is that all vertices of $V_H$ live on the same branch of $T$. 

To prove Lemma~\ref{lem:main cut-dense union} we show that there exists a pair of vertices $u,v\in V(G)$ such that $u$ and $v$ are at $T$-distance at least $2d/\beta'$ and there also is a path from $u$ to $v$ in $G$ of length at most $2/\beta'$. We can then conclude that there is an edge with $T$-length at least $d$ in $G$ by following this path from $u$ to $v$.
The following lemma and Lemma~\ref{lem: path cant split} capture the ideas we use to show that there are vertices with large enough $T$-distance. 
\begin{lemma}\label{lem: long path in T}
    Let $\mathcal H'\subseteq \mathcal H$ be such that $|\mathcal{H}'|\geq 16/(\lambda \beta'^2 )$ and such that $w_H$ and $w_{H'}$ are comparable for all $H,H'\in\mathcal H'$. Then, there are $H^-,H^+\in \mathcal H'$ such that $f_{H^-}<_T w_{H^+}$ and $f_{H^-}$ and $w_{H^+}$ have $T$-distance at least $4d/\beta' $.
\end{lemma}
\begin{proof}
    Recall that, by hypothesis \ref{hyp:gamma} of Lemma~\ref{lem:main cut-dense union}, for every $H$ in $\mathcal H'$ there are at most $1/\lambda$ many $H'\in\mathcal H'$ intersecting $H$ in more than $\gamma d$ vertices. By Lemma~\ref{lem: independent set}, there exists $\mathcal H''\subseteq\mathcal H'$ of size at least $|\mathcal H'|/(1/\lambda +1)\geq 8/\beta'^2$ such that no graphs in $\mathcal H''$ intersect in more than $\gamma d$ vertices. By throwing away some graphs, let us assume that the size of $\mathcal H''$ is exactly $8/\beta'^2$. 
    
    By Observation~\ref{key observation}, $f_H$ and $w_{H'}$ as well as $f_H$ and $f_{H'}$ are comparable for all $H,H'\in\mathcal H''$. Let $H^-\in\mathcal H''$ be such that $f_{H^-}$ is minimal and $H^+\in\mathcal H''$ such that $w_{H^+}$ is maximal, noting in particular that we then have $f_{H^-}<_T w_{H^+}$ (with the possibility that $H^-=H^+$). Moreover, the path in $T$ from $f_{H^-}$ to $w_{H^+}$ contains all of $\bigcup_{H\in\mathcal H''}V_{H}$. 
    
    For any $H\in\mathcal H''$, we have
    $$
    \sum_{H'\in(\mathcal H''-H)}|V(H)\cap V(H')|\leq 8/\beta'^2\cdot\gamma d\leq \beta' d/2,
    $$
    using that $\gamma \ll \beta'$.
    We get $|V_{H}\setminus (\bigcup_{H'\in (\mathcal H''-H)}V_{H'})|\geq \beta'  d/2$ (recall $|V_H|\geq \beta' d$), from which we conclude $|\bigcup_{H\in\mathcal H''}V_{H}|\geq 8/\beta'^2\cdot \beta' d/2= 4d/\beta'$. Thus, $f_{H^-}$ and $w_{H^+}$ are at $T$-distance at least $4d/\beta' $.
\end{proof}
 
We say that $H,H'\in\mathcal H$ \textit{intersect robustly} if they share a vertex outside of $V_H\cup V_{H'}$. Let $F$ be an auxiliary graph whose vertices are the elements of $\mathcal H$ and we join $H,H'\in\mathcal H$ by an edge if and only if they intersect robustly. The key to the proof of Lemma~\ref{lem:main cut-dense union} is to use Lemma~\ref{lem: average degree in cluster graph after removal} and Lemma~\ref{lem: get path from average degree} to understand the graph $F$ and combine this with the knowledge we are currently building up about the interaction of the graphs in $\mathcal H$ with $T$. 

Next, we provide two lemmas which help us understand how $T$ interacts with $F$ and $\mathcal H$, more specifically, with the vertices $w_H$ for $H\in\mathcal H$.
\begin{lemma}\label{lem: w's are comparable simple}
    If $H,H'\in\mathcal H$ intersect robustly, then $w_H$ and $w_{H'}$ are comparable with respect to $T$.
\end{lemma}
\begin{proof}
    By definition, there is some $v \in V(H)\cap V(H')\setminus (V_H \cup V_{H'})$. As $v \in V(H) \setminus V_H$, we know that $v$ is not one of the first $\beta' d$ vertices of $H$, hence $w_H<_T v$ . By the same logic, $w_{H'}<_T v$. By Observation~\ref{key observation}, we conclude that $w_H$ and $w_{H'}$ are comparable.
\end{proof}

The next lemma generalizes Lemma~\ref{lem: w's are comparable simple} to connected components in $F$. Given a family $\mathcal H'\subseteq\mathcal H$, we write $w(\mathcal H')\subseteq V(G)$ to denote the set $\{w_H:H \in \mathcal H'\}$.
\begin{lemma}\label{lem: w's are comparable}
    Let $\mathcal H'\subseteq \mathcal H$ be such that $F[\mathcal H']$ is connected. Then, the root of $T\llbracket {w(\mathcal H')}\rrbracket$ is in $w(\mathcal H')$. 
\end{lemma}
\begin{proof}
    Let $v$ be the root of $T\llbracket {w(\mathcal H')}\rrbracket$ and suppose towards a contradiction that $v\notin w(\mathcal H')$. As $v$ is the root of $T\llbracket {w(\mathcal H')}\rrbracket$, there exists a child $u$ of $v$ in $T$ such that $T_u$ contains some vertices of $w(\mathcal H')$ but not all. As $F[\mathcal H']$ is connected there exist $H,H'\in\mathcal H'$ with $HH'\in E(F)$ such that $w_H\in V(T_u)$ and $w_{H'}\notin V(T_u)$. By the definition of $F$, we know that $H$ and $H'$ intersect robustly. So, by Lemma~\ref{lem: w's are comparable simple}, we get that $w_H$ and $w_{H'}$ are comparable with respect to $T$. But then we have that $w_{H'}<_T u$ and since, by assumption, $w_{H'}\neq v$, we get that $w_{H'}\notin T_v$. This is a contradiction to the assumption that $v$ is the root of $T\llbracket {w(\mathcal H')}\rrbracket$, which implies that $T\llbracket {w(\mathcal H')}\rrbracket\subseteq T_v$.
\end{proof}

\par Finally, we require one more lemma providing us with vertices at large enough $T$-distance. Given a vertex $v\in V(G)$, we write $\mathcal H_v\subseteq\mathcal H$ for the set of $H\in\mathcal H$ such that $v<_T w_H$.

\begin{lemma}\label{lem: path cant split}    
    Let $P$ be a path in $F$, and suppose that $v\in V(G)$ is a vertex such that $P[\mathcal H_v]$ has at least $16/(\lambda  \beta'^2)+1$ components. Then there is an edge $HH'\in E(F)$ such that $f_H\leq_T w_{H'}$ and $f_H$ and $w_{H'}$ are at $T$-distance at least $4d/\beta' $.
\end{lemma}
\begin{proof}
    Since $P[\mathcal H_v]$ has at least $16/(\lambda \cdot \beta'^2 )+1$ components, we can choose distinct $$H_1,H_1',\ldots,H_{16/(\lambda \cdot \beta'^2 )},H_{16/(\lambda \cdot \beta'^2 )}'$$ such that $H_iH_i'\in E(P)$ with $H_i\in \mathcal H_v$ and $H_i'\notin \mathcal H_v$. These choices are made by scanning along the path $P$, and setting $H_i$ as the last vertex of the $i$th component of $P[\mathcal H_v]$, and $H_i'$ as the next vertex on the path. As $H_i$ and $H_i'$ intersect robustly, by Lemma~\ref{lem: w's are comparable simple} $w_{H_i}$ is comparable with $w_{H_i'}$, and hence $w_{H_i'}$ and $v$ are comparable by Observation~\ref{key observation}. As $H_i'\notin \mathcal H_v$, we must have $w_{H_i'}\leq_T v$. Let $W$ be the set of $w_{H_i'}$ and note that vertices of $W$ are pairwise comparable again by Observation~\ref{key observation}. By Lemma~\ref{lem: long path in T}, there is $i,j$ such that $f_{H_i'}$ is at $T$-distance at least $4d/\beta' $ from $w_{H_j'}$ with $f_{H_i'}<_Tw_{H_j'}$. Since $w_{H_j'}\leq_T v\leq_T w_{H_i}$, $f_{H_i'}$ is also at $T$-distance at least $4d/\beta' $ from $w_{H_i}$, completing the proof since $H_iH_i'\in E(F)$ by assumption.
\end{proof}
We can now combine all of these tools to prove the main result of this section.

\begin{proof}[Proof of \cref{lem:main cut-dense union}]
    Our goal is to show there are $H,H'\in\mathcal H$, either equal or with $HH'\in E(F)$, such that $f_H$ is at $T$-distance at least $2d/\beta' $ from $w_{H'}$. To see that this is sufficient, recall the diameter of $H$ and $H'$ is at most $1/\beta' $ by Lemma~\ref{lem: cut-dense properties}, so we can then deduce there is a path in $G$ of length at most $2/\beta' $ from $f_H$ to $w_{H'}$. It then follows that some edge on this path has $T$-length at least $d$. So it remains to achieve the stated goal.
    
    With the aim of applying Lemma~\ref{lem: average degree in cluster graph after removal}, consider $F$ from the following perspective. First, we create a hypergraph on $V(G)$ which contains a hyperedge on $V(H)$ for every $H\in\mathcal H$. Then, we remove from the hyperedge $V(H)$ the vertices in $V_H$. Finally, we take the intersection graph of this hypergraph and observe that it is exactly $F$. Therefore, $F$ has average degree at least $\lambda/(8\gamma)$ by Lemma~\ref{lem: average degree in cluster graph after removal}, where we use that $\lambda\gg\beta',\zeta$. By Lemma~\ref{lem: get path from average degree}, $F$ contains a vertex $H^*$ together with a path $P$ containing at least $\lambda /(16\gamma)$ neighbors of $H^*$. Let $\mathcal N$ be the set of these neighbors. 
    
    By Lemma~\ref{lem: w's are comparable simple}, for each $H\in\mathcal N$, $w_H$ is comparable to $w_{H^*}$. Let $\mathcal N'\subseteq \mathcal N$ be the set of $H\in\mathcal N$ with $w_H\leq_T w_{H^*}$. By Observation~\ref{key observation}, $w_H$ and $w_{H'}$ are comparable for all $H,H'\in\mathcal N'$.
    If $|\mathcal N'|\geq 16/(\lambda \cdot \beta'^2 )$ then by Lemma~\ref{lem: long path in T} there are $H,H'\in\mathcal N'$ with $f_H<_T w_{H'}$ at $T$-distance at least $4d/\beta' $ from each other. Since $w_{H'}\leq_T w_{H^*}$, it follows that $f_H$ is at $T$-distance at least $4d/\beta' $ from $w_{H^*}$, concluding the proof.
    
    Therefore, we may assume that $|\mathcal N'|< 16/(\lambda \cdot \beta'^2 )$ implying that $P[\mathcal H_{w_{H^*}}]$ contains at least $\lambda/(16\gamma)-16/(\lambda \beta'^2) \geq\lambda /(32\gamma)$ graphs of $\mathcal N$ (here we used that $\gamma\ll \lambda,\beta'$). Additionally, by Lemma~\ref{lem: path cant split}, we may assume $P[\mathcal H_{w_{H^*}}]$ has at most $16/(\lambda \beta'^2 )$ components, otherwise we are done. Thus, one component, say $P_1$, must contain at least $\lambda /(32\gamma)\cdot \lambda  \beta'^2 /16$ graphs of $\mathcal N$. 
    \par By Lemma~\ref{lem: w's are comparable}, the root $w_1$ of $T\llbracket w(V(P_1))\rrbracket$ is itself in $w(V(P_1))$. Let $H_1\in V(P_1)$ be such that $w_{H_1}=w_1$. 
        \par We may assume there are at most $16/(\lambda  \cdot \beta'^2 )$ graphs $H\in V(P_1)$ such that $w_1=w_H$. To see this, suppose the contrary, and note that such $H$ have the collection of their $w_H$ being (trivially) comparable. Thus, by Lemma~\ref{lem: long path in T}, we can conclude that there exists some $H$ so that $f_H$ and $w_H$ (where $w_H=w_1$) are of $T$-distance at least $4d/\beta'$, which allows us to achieve our goal.
    \par The members $H\in V(P_1)$ for which $w_H$ is not equal to $w_1$ must have $w_H\geq_T w_1$. Therefore, the previous paragraph allows us to assume that $P_1[\mathcal H_{w_1}]$ contains at least $$(\lambda/(32\gamma))\cdot (\lambda \beta'^2/16) - 16/(\lambda\beta'^2) \geq \lambda /(32\gamma)\cdot (\lambda  \beta'^2 /16)\cdot (1/2)$$ members of $\mathcal N$, again using that $\gamma \ll \lambda,\beta'$. Furthermore, $P_1[\mathcal H_{w_1}]$ has, by Lemma~\ref{lem: path cant split}, at most $16/(\lambda \cdot \beta'^2 )$ components, as otherwise we again achieve our goal. By averaging, we can thus conclude that one component of $P_1[H_{w_1}]$, say $P_2$, must contain at least $\lambda /(32\gamma)\cdot (\lambda \beta'^2 /16)^2\cdot 1/2$ graphs of $\mathcal N$. Since $\gamma$ is sufficiently small compared to $\lambda $ and $\beta' $, we can iterate this process $16/(\lambda  \beta'^2 )$ times (we can iterate $i$ times as long as $\lambda /(32\gamma)\cdot (\lambda  \beta'^2 /16)^{i+1}\cdot (1/2)^i\geq 32/(\lambda\beta'^2)$), obtaining a sequence $H_1,\ldots,H_{16/(\lambda\beta'^2)}\in V(P)$ satisfying $w_{H^*}<_T w_{H_1}<_T \dots<_T w_{H_{16/(\lambda \cdot \beta'^2 )}}$ such that $\mathcal H_{w_{H_{16/(\lambda \cdot \beta'^2 )}}}$ contains some $H\in\mathcal N$. Note that then $w_{H_{16/(\lambda \cdot \beta'^2 )}}<_T w_H$.
    
    By Lemma~\ref{lem: long path in T}, there is $i$ such that $f_{H_i}$ is at $T$-distance at least $4d/\beta' $ from some $w_{H_j}$, and hence at $T$-distance at least $4d/\beta'$ from $w_{H_{16/(\lambda \cdot \beta'^2 )}}$. Then, either $f_{H_i}$ is at $T$-distance at least $2d/\beta' $ from $w_{H_i}$, or otherwise we must have that $w_{H_i}$ is at $T$-distance at least $2d/\beta' $ from $w_{H_{16/(\lambda \cdot \beta'^2 )}}$. In the first case, we achieve our goal with $H=H'=H_i$, so suppose that we are in the second case. Then, $w_{H^*}$ and $w_H$ are also at $T$-distance at least $2d/\beta' $ by the chain of inequalities above. Of course, we also have that $f_{H^*}\leq_T w_{H^*}$. As $H\in\mathcal N$, we get $H^*H\in E(F)$, so we have achieved our goal.
\end{proof}

\section{Proof of main results}\label{sec:final}
In this section we provide proofs of Theorem~\ref{thm: main} and Theorem~\ref{thm:EG hyperstability}. While the two proofs are somewhat similar, we present the proof of Theorem~\ref{thm:EG hyperstability} first as it is a bit simpler.

\subsection{Proof of Theorem~\ref{thm:EG hyperstability}}
First, let us recall the theorem statement.
\vspace{2mm}

\noindent \textbf{Theorem \ref{thm:EG hyperstability}. }\textit{For every $c>0$ and every sufficiently large $d$, the following holds for every $n$-vertex graph $G$. Either $G$ contains a cycle of length at least $d$, or else one can delete at most $c dn$ edges from $G$ to obtain a graph in which every connected component has a vertex-cover of order at most $d$.}
\vspace{2mm}

\noindent We will use the following simple lemma.
\begin{lemma}\label{lem: remove o(1)}
    The following holds for every $c',d>0$ and every graph $G$ with a vertex-cover of size $(1+c')d$. One can delete at most $2c'e(G)$ edges from $G$ to obtain a graph with vertex-cover of size at most $d$.
\end{lemma}
\begin{proof}
    Let $V\subseteq V(G)$ be a minimum vertex-cover of $G$ and set $d'=|V|$. If $d'\leq d$ then we are done, so suppose that $d'>d$. Let $V'\subseteq V$ be obtained by selecting $d'-d$ vertices uniformly at random. Finally, we obtain $G'\subseteq G$ by deleting all edges incident to $V'$ from $G$. Then, $G'$ has a vertex-cover of size at most $d$, namely $V\setminus V'$, so that it remains to show that the expected number of deleted edges is at most $2c'e(G)$. The probability that a specific vertex is in $V'$ is at most $\frac{d'-d}{d}\leq c'$. Therefore, any edge is incident to $V'$ with probability at most $2c'$. The bound follows by the linearity of expectation. 
\end{proof}

\begin{proof}[Proof of Theorem~\ref{thm:EG hyperstability}]
    We begin by explaining the choice of parameters, which is somewhat intricate. We are given $c>0$ and set $c'=c/3$. Set $\varepsilon=\eta=2c'/3$, $C=1+c'$, $J=1+\eta/8$ and $\xi = 1-(1+c')^{-1}$. We then obtain $\delta$ from Lemma~\ref{lem: nowhere dense case} and, subsequently, $\lambda,q,M$ from Theorem~\ref{thm: structure}. Next, we get $\zeta$ through Lemma~\ref{lem:main cut-dense union}, then $\beta$ from Theorem~\ref{thm: structure} and $\gamma$ from Lemma~\ref{lem:main cut-dense union}. Finally, we obtain $\rho$ from Theorem~\ref{thm: structure}. Let $d$ be large enough compared to all these parameters.

If $e(G)\geq dn/2$ then $G$ contains a cycle of length at least $d$ by the Erd\H{o}s--Gallai theorem. Therefore, let us assume that $e(G)<dn/2$. Let $T\subseteq G$ be a DFS-forest of $G$ and set $G'=G\setminus T$. Apply \cref{thm: structure} to write $G'=G_0\cup G_1\cup G_2\cup G_3\cup G_4$, and note that then $G=T\cup G_0\cup G_1\cup G_2\cup G_3\cup G_4$. Observe that every component of $G_1$ has a vertex-cover of size at most $Cd=(1+c')d$. By Lemma~\ref{lem: remove o(1)} applied to every component of $G_1$, we can delete at most $2c'e(G_1)$ edges from $G_1$ to obtain a graph in which every connected component has a vertex-cover of size at most $d$. Thus, it suffices to show that, under the assumption that $G$ does not contain a cycle of length at least $d$, $T,G_0$ and $G_3$ each contain at most $c'dn/3$ edges, and $G_2$ and $G_4$ are both empty.
\par Note that $T$ is a forest so $T$ contains at most $n-1\leq c'dn/3$ edges, as $d\gg c'^{-1}$.
Theorem~\ref{thm: structure} guarantees that $G_0$ contains at most $\varepsilon e(G)\leq \varepsilon dn/2\leq c'dn/3$ edges. 

Moving forward, suppose $G_2$ is not empty. Then there exists $H\subseteq G_2$ which is $\rho$-cut-dense and contains a $qd$-regular subgraph of order at least $Cd$ and $|V(H)|\leq Md$. By Lemma~\ref{lem:cut-dense case} applied with $p=1$, $H$ contains a cycle of length at least $(1-\xi)Cd= d$ and, therefore, so does $G$.

Next, if $G_3$ contains more than $c'dn/3$ edges then $G_3$ has average degree at least $\eta d$. By Lemma~\ref{lem: nowhere dense case}, it would then follow that $G_3$ contains an edge of $T$-length at least $d$. Consequently, $G$ would contain a cycle of length at least $d$. Hence, we can conclude that $G_3$ contains less than $c'dn/3$ edges. 

 Finally, suppose that $G_4$ is not empty. By Lemma~\ref{lem:main cut-dense union}, $G_4$ would then contain an edge of $T$-length at least $d$ implying that $G$ contains a cycle of length at least $d$, hence we can also suppose this is not the case, concluding the proof.
\end{proof}

\subsection{Proof of Theorem~\ref{thm: main}}
We reproduce Theorem~\ref{thm: main} below for the reader's convenience.
\vspace{2mm}

\noindent \textbf{Theorem \ref{thm: main}. }\textit{For every $c>0$ there exists $K$ such that the following holds for all $d\geq K$ and $p\geq K/d$. Let $G$ be a graph on $n$ vertices with average degree $d$. Then, with probability tending to $1$ as $n\to \infty$, $G_p$ contains a cycle of length at least $(1-c)d$.}

\begin{proof}[Proof of Theorem~\ref{thm: main}]
First, let us elaborate on the (unfortunately) quite complicated way of choosing the parameters. We note that almost all parameters in lemmas and theorems referred to in this proof correspond to the parameter here with the same label. The only exceptions to this are the application of Lemma~\ref{lem: cover to dense} which we will apply with $(1-\tau)d$ in place of $d$ as well as with $C/(1-\tau)$ instead of $C$ and Lemma~\ref{lem: dense case}, which we will apply with $(1-2\tau)d$ instead of $d$. To begin with, we are given $c>0$. Set $\tau = c/3$, $\varepsilon=\eta=\tau/3$, $C=2$, $J=1+\eta/8$ and $\xi = 1/2$. Let $\delta>0$ be as in Lemma~\ref{lem: nowhere dense case} and $\ell$ from Lemma~\ref{lem: cover to dense}. We then obtain $\lambda,q,M>0$ from Theorem~\ref{thm: structure}. Next, we get $\zeta>0$ through Lemma~\ref{lem:main cut-dense union}. From Theorem~\ref{thm: structure}, we then obtain $\beta>0$, with which we get $\gamma>0$ from Lemma~\ref{lem:main cut-dense union}. Next, we let Theorem~\ref{thm: structure} define $\rho$. Let $K$ be large enough compared to all these parameters, $d\geq K$ and $p\geq K/d$. Let $G$ be a graph on $n$ vertices with average degree $d$ and note that this implies $n\geq d$.

Let us run the percolated DFS-algorithm on $G$, producing, by Fact~\ref{fact: DFS-forest}, a DFS-forest $T\subseteq G_p$ of the graph $G'$ consisting of the edges that are not queried. By Fact~\ref{fact: DFS edges revealed}, at most $n-1$ edges are used by the percolated DFS-algorithm. Recall that the percolated DFS-algorithm uses an edge $e$ if and only if it queries $e$ and $e$ is present in $G_p$. We can couple the outcomes of the queries of the percolated DFS-algorithm with an infinite sequence $X_1,X_2,\ldots$ of independent Bernoulli random variables with parameter $p$ such that the $i$th queried edge is used if and only if $X_i=1$. Let $Q$ denote the number of queried edges. The earlier observation that at most $n-1$ edges get used then translates to $\sum_{i=1}^QX_i\leq n-1$. A standard application of the Chernoff bound together with a union bound over all $Q\geq 2dn/K$ shows that with high probability (as $n\to \infty$), $Q< 2dn/K$. For the remainder of the argument, let us assume this is the case. Therefore, we may assume that $G'$ contains at least $(1-4/K)dn/2$ edges. Note that we have not touched the randomness of the edges in $G'$, so that each edge of $G'$ is present in $G_p$ with probability $p$ independently of all other edges. Let $L,H_1,\ldots,H_r,F_1,\ldots,F_h$ be a collection of edge-disjoint subgraphs of $G'$ maximizing $(|E(L)|,r,h)$ lexicographically such that
\begin{itemize}
    \item every edge in $L$ has $T$-length at least $d$,
    \item each $H_i$ is $\rho$-cut-dense, contains a $qd$-regular subgraph of order at least $Cd$ and $|V(H_i)|\leq Md$, and,
    \item each $F_i$ is of size at most $\ell d$ and has average degree at least $(1-2\tau)d$.
\end{itemize}
 Next, we argue that if any of $L$, $\bigcup_{i=1}^rH_i$ or $\bigcup_{i=1}^h F_i$ contains at least $dn/K$ edges then $G_p$ contains a cycle of length at least $(1-3\tau)d=(1-c)d$ with high probability. To see this we treat the three cases separately. Suppose that $|E(L)|\geq dn/K$. Since every edge in $L$ produces a cycle of length at least $d$ when added to $T\subseteq G_p$, we get that $G_p$ contains a cycle of length at least $d$ with probability at least $1-(1-p)^{|E(L)|}\geq 1-e^{-p|E(L)|}$. Recall that $p\geq K/d$, so this probability is at least $1-e^{-n}$. Next, let us consider the situation when $\bigcup_{i=1}^rH_i$ contains at least $dn/K$ edges. Each $H_i$ contains at most $M^2d^2$ edges so that $r\geq \frac{dn/K}{M^2d^2}= n/(KM^2d)$. By Lemma~\ref{lem:cut-dense case}, each $H_i$ yields a cycle of length at least $d$ in $G_p$ with probability at least $1-e^{-\Omega(d)}$. Note that these events are independent for different $H_i$ because they are edge-disjoint. Therefore, $G_p$ contains a cycle of length at least $d$ with probability at least $1-e^{-\Omega(n)}$, where we treat $M$ and $K$ as constants. Finally, suppose that $\bigcup_{i=1}^h F_i$ contains at least $dn/K$ edges. We proceed with the same idea as in the previous case. Each $F_i$ contains at most $\ell^2d^2$ edges, so that $h\geq n/(K\ell^2d)$. By Lemma~\ref{lem: dense case}, each $F_i$ retains a cycle of length at least $(1-\tau)(1-2\tau)d\geq (1-3\tau)d$ in $G_p$ with probability at least $1-e^{-\Omega(d)}$. Noting again that this is independent for each $F_i$, we get a cycle of length at least $(1-3\tau)d$ in $G_p$ with probability at least $1-e^{-\Omega(n)}$, where we treat $\ell$ and $K$ as constants.
 
 Let $R=G'-\left(L\cup\bigcup_{i=1}^rH_i\cup\bigcup_{i=1}^h F_i\right)$. As a consequence of the above arguments, we may assume that $R$ contains at least $(1-10/K)dn/2$ edges. Let $R=G_0\cup G_1\cup G_2\cup G_3\cup G_4$ be an edge-decomposition as guaranteed by Theorem~\ref{thm: structure}.
 
 By assumption, $G_0$ contains at most $\varepsilon dn/2$ edges. $G_2$ must be empty so as not to contradict the maximality of $r$. By Lemma~\ref{lem: nowhere dense case}, $G_3$ contains at most $\eta dn/2$ edges as otherwise there is an edge of $T$-length at least $d$ in $R$, contradicting the maximality of $|E(L)|$. Similarly, by Lemma~\ref{lem:main cut-dense union}, $G_4$ is empty. Finally, $G_1$ can contain at most $(1-\tau)dn/2$ edges. Indeed, suppose otherwise. By averaging there is a connected subgraph $H\subseteq G_1$ of average degree at least $(1-\tau)d$ with vertex-cover of size at most $Cd$. By Lemma~\ref{lem: cover to dense}, there exists $H'\subseteq H$ on at most $\ell d$ vertices with average degree at least $(1-\tau)^2d\geq (1-2\tau)d$, contradicting the maximality of $h$. We conclude that 
 $$
 e(R)\leq e(G_0)+e(G_1)+e(G_2)+e(G_3)+e(G_4)\leq \varepsilon dn/2+(1-\tau)dn/2+\eta dn/2<(1-10/K)dn/2,
 $$
 where the last inequality uses $\varepsilon+\eta=2\tau/3$ and $\tau\gg1/K$. This inequality is in direct contradiction to our earlier observation of the opposite inequality.
 \end{proof}

\section{Concluding remarks}\label{sec: concluding}

We begin by reiterating a problem raised by Krivelevich, Lee, and Sudakov~\cite{MR3302900}: is it true that if $G$ has minimum degree $d$ and $p\geq \frac{(1+\eps)\log d }{d}$, $G_p$ almost surely contains a cycle of length $d+1$? Glebov, Naves and Sudakov~\cite{glebov2017threshold} proposed a more refined version of this problem, where minimum degree is replaced with average degree, and the assumption on $p$ is replaced with $p=\frac{\log d+\log\log d+\omega(1)}{d}$. We believe that the methods in this paper offer a clear line of attack for these problems; however, it seems likely that difficulties would arise in their execution. 

Our methods are also likely to offer some leverage for embedding trees in percolated graphs, especially when combined with the techniques used by Pokrovskiy \cite{alexey}. We raise the following conjecture, which is a robust variant of (an approximate form of) the Erd\H os--S\' os conjecture.

\begin{conjecture}\label{conj: robus erdos sos alt}
    For every $c>0$ the following holds for all large enough $K$. Let $G$ be a graph on $n$ vertices with average degree $d$ and let $T$ be a tree on $(1-c)d$ vertices with maximum degree at most $\Delta$. Let $p\geq K \Delta /d$. Then, with probability tending to $1$ as $n\rightarrow\infty$, $G_p$ contains a copy of $T$.
\end{conjecture} We emphasize that $K$ depends only on $c$ in the above statement. When $p=1$, the above conjecture becomes a deterministic statement that is an approximate version of the Erd\H os--S\' os conjecture that is open, to the best of our knowledge. When $p=1$ and $d=\Omega(n)$, however, a result of Davoodi, Piguet, \v{R}ada and Sanhueza-Matamala \cite{davoodi2026asymptotic} confirms the conjecture. When $d=n-1$, the conjecture is simply an assertion about the binomial random graph $G(n,p)$, which holds due to a result of Montgomery~\cite[Theorem 1.4]{montgomery2014sharp} (that builds on prior work \cite{balogh2010large, alon2007embedding}). Given the apparent difficulty of Conjecture~\ref{conj: robus erdos sos alt}, weaker forms might be of interest. For example, it may already be interesting to consider the case when $T$ is a bounded-degree tree. 
 \paragraph{Acknowledgements:}We would like to thank Raphael Steiner for fruitful discussions in the early stages of the project. We are also grateful to Alexey Pokrovskiy for helpful discussions regarding his structure theorem. Part of this research was carried out while the second author was visiting the Institute for Theoretical Studies at ETH Z\"urich, and we thank them for providing a wonderful working environment.

 \paragraph{Statement of AI use:} ChatGPT 5.6 was used to proofread this paper. All of the ideas and writing are due entirely to the authors.

\bibliographystyle{yuval}
{{\bibliography{lib}}}

\appendix
\section{Proof of Theorem~\ref{thm: structure}}\label{sec:appendix}
The goal of this appendix is to give a proof of Theorem~\ref{thm: structure}. First, let us recall the statement.
\vspace{2mm}

\noindent \textbf{Theorem \ref{thm: structure}. }\textit{
    For every $C,J,\varepsilon,{\delta}>0$, there exist $\lambda,q,M>0$, such that for every $\zeta>0$, there exists $\beta>0$ such that for every $\gamma>0$ there exists $\rho>0$ such that for every sufficiently large $d$ and every graph $G$, there is an edge-decomposition of $G$ into $G_0,G_1,G_2,G_3,G_4$ satisfying the following properties
    \begin{description}
        \item[Exceptional edges] The graph $G_0$ contains at most $\varepsilon e(G)$ edges.
        \item[Small vertex-cover] Every component of $G_1$ has a vertex-cover of size at most $Cd$. 
        \item[Large regular cores] $G_2$ is the union of edge-disjoint $\rho$-cut-dense graphs $H$, each of which contains a $qd$-regular subgraph of order at least $Cd$, and $\ab{V(H)}\leq Md$.  
        \item[Nowhere dense] $G_3$ does not contain a set of at most $Jd$ vertices inducing at least ${\delta} d^2$ edges.
        \item[Well-connected pieces] $G_4$ is the edge-disjoint union of a collection $\mathcal H$ of graphs satisfying the following properties.
        \begin{enumerate}
            \item Each $H\in\mathcal H$ contains a subgraph $Q$ of size at least $\beta d$ which is $\beta$-cut-dense and every vertex in $V(H)\setminus V(Q)$ has at least $\beta d$ neighbors in $V(Q)$.
            \item Any two $H,H'\in\mathcal H$ intersect in at most $\zeta d$ vertices.
            \item Every $H\in \mathcal H$ has at least $\lambda d$ vertices which intersect some other $H'\in\mathcal H$.
            \item For every $H\in\mathcal H$ there are at most $1/\lambda$ other $H'\in\mathcal H$ which intersect $H$ in more than $\gamma d$ vertices.
        \end{enumerate}
    \end{description}}

\subsection{Clump-theoretic machinery}
Since the proof of Theorem~\ref{thm: structure} closely follows the argument of Pokrovskiy \cite{alexey}, let us begin by quoting, more or less verbatim, many of the auxiliary definitions and results of \cite{alexey}. 

\begin{definition}[{\cite[Definition 4.1]{alexey}}]\label{pok4.1}
    Let $p,\kappa>0$ and let $m$ be an integer. A \emph{clump} with parameters $p,\kappa,m$ is a graph $K$, as well as families $\mathcal H(K), \mathcal M(K)$, a graph $C(K)$, and an integer $k(K)$, satisfying the following properties.
    \begin{enumerate}[label=(\Alph*),start=8]
        \item $\mathcal H(K)$ is a collection of subgraphs $H_1,\dots,H_h \subseteq K$, each of which has $m$ vertices and is $p$-cut-dense, and which form an edge-partition of $K$. That is, $E(H_i)\cap E(H_j)=\varnothing$ for all $i \neq j$, and $E(K)=\bigcup_{i=1}^h E(H_i)$.\label{pok(H)}

        \setcounter{enumi}{12}
        \item $\mathcal M(K)$ is a collection of exactly $k(K)$ vertex-disjoint graphs $M_1,\dots,M_{k(K)}$, each of which is $p^{13}m$-regular. Moreover, each $M_i$ is a subgraph of some $H \in \mathcal H(K)$. Finally, $k(K)$ is as large as possible subject to this constraint; namely, having fixed $\mathcal H(K)$, there is no collection of $k(K)+1$ vertex-disjoint $p^{13}m$-regular graphs, each of which is a subgraph of some $H \in \mathcal H(K)$. 

        We also let $M(K)=\bigcup_{i=1}^{k(K)}M_i$ denote the union of all graphs appearing in $\mathcal M(K)$.\label{pok(M)}

        \setcounter{enumi}{2}
        \item $C(K)$ is a subgraph of $K$ with at most $4^{k(K)}m$ vertices which is $\kappa^{(10k(K))!}$-cut-dense, and such that $M(K)$ is a subgraph of $C(K)$.\label{pok(C)}
    \end{enumerate}
    Next, given a clump $K$, we define a subgraph $B(K)$ with edge set
    \[
    E(B(K)) = E(M(K)) \cup \{uv\in E(K) : u \in V(M(K)), v \in V(K) \setminus V(M(K)) , \ab{N_K(v) \cap V(M(K))}\geq pm/2\},
    \]
    and whose vertex set is all vertices of $K$ which are incident to at least one edge of $E(B(K))$. 

    Similarly, we let $D(K)$ be the subgraph of $K$ defined as $D(K)=C(K)\cup B(K)$. In particular, by \ref{pok(C)}, we have
    \[
    E(D(K)) = E(C(K)) \cup \{uv : u \in V(M(K)), v \in V(K) \setminus V(M(K)) , \ab{N_K(v) \cap V(M(K))}\geq pm/2\}.
    \]
\end{definition}
We stress that $K,M(K),C(K),B(K),D(K)$ are all graphs, rather than merely sets of vertices. Below some basic properties of these graphs proved in \cite{alexey} are stated. 

\begin{lemma}[{\cite[Fact 4.2]{alexey}}]\label{pok4.2}
    If $K$ is a clump, then $V(M(K))$ is a cover of $B(K)$. 
\end{lemma}

\begin{lemma}[{\cite[Fact 4.3]{alexey}}]\label{pok4.3}
    If $K$ is a clump with parameters $p,\kappa,m$, then $k(K) \leq \ab{V(M(K))}/(p^{13}m)$.
\end{lemma}

\begin{lemma}[{\cite[Lemma 4.4]{alexey}}]\label{pok4.4}
    If $K$ is a clump with parameters $p,\kappa,m$, then $e(K)\geq p^{13} m\ab{V(B(K))}/10$.
\end{lemma}

The following shows that most edges of a clump live in $B(K)$.
\begin{lemma}[{\cite[Lemma 4.6]{alexey}}]\label{pok4.6}
    Let $p,\kappa \in (0,\frac{1}{100})$ and let $m$ be sufficiently large with respect to $p$. Let $K$ be a clump with parameters $p,\kappa,m$, and let $H \in \mathcal H(K)$. Then $\ab{E(H)\setminus E(B(K))}\leq p^2 m^2$, and $\ab{E(K)\setminus E(B(K))}\leq 4pe(K)$.
\end{lemma}
The following lemma gives a sense in which clumps, as well as certain unions of them, are well-connected.
\begin{lemma}[{\cite[Lemma 4.7]{alexey}}]\label{pok4.7}
    Let $0 <\gamma,\kappa \leq p^{13} \leq 10^{-3}$, and let $m$ be sufficiently large with respect to $\gamma$. Let $K_1,\dots,K_t$ be non-empty clumps with parameters $p,\kappa,m$, and let $k \geq \max_i k(K_i)$ be an integer. For all $i \in [t]$, let $I_i \subseteq V(D(K_1)) \cap V(D(K_i))$ be a set of vertices with $\ab{I_i}=\gamma m$. Additionally, for all $i \in [t]$, let $H_i$ be a graph such that either $H_i=\varnothing$ or $H_i \in \mathcal H(K_i)$. Then the graph $$\bigcup_{i=1}^t (D(K_i)[I_i \cup V(C(K_i))] \cup H_i)$$ is $\kappa^{(10k)!}\gamma^{2kt^2}$-cut-dense, where $D(K_i)[S]$ denotes the subgraph of $D(K_i)$ induced on a vertex set $S$.
\end{lemma}

To start a clumping argument, we need to be able to identify some clumps in the first place, assuming that the graph has a dense spot. The following lemma shows that once we find a cut-dense graph, we can give it the structure of a clump.
\begin{lemma}[{\cite[Lemma 4.8]{alexey}}]\label{pok4.8}
    Let $0 <\kappa \leq p \leq 10^{-3}$, and let $m$ be sufficiently large with respect to $\kappa$. Let $H$ be a $p$-cut-dense graph on $m$ vertices. Then there is a clump $K$ with $C(K)=H,\mathcal H(K)=\{H\}$, parameters $p,\kappa,m$, and satisfying $1 \leq k(K) \leq p^{-13}$. 
\end{lemma}

The following lemma is the main engine behind the clumping argument.

\begin{lemma}[{\cite[Lemma 4.9]{alexey}}]\label{pok4.9}
    Let $0 <\kappa \leq p \leq 10^{-3}$, and let $m$ be sufficiently large with respect to $\kappa$. Let $K^+,K^-$ be edge-disjoint clumps with parameters $p,\kappa,m$, and suppose that $k(K^+),k(K^-)\leq \kappa^{-1}$. Suppose further that $\ab{V(D(K^+)) \cap V(D(K^-))}\geq \kappa^{6(10\max\{k(K^+),k(K^-)\})!}m$. Then there exists a clump $K$ with parameters $p,\kappa,m$ satisfying $E(K) = E(K^+)\cup E(K^-), \mathcal H(K) = \mathcal H(K^+)\cup \mathcal H(K^-)$, and $\max\{k(K^+),k(K^-)\} \leq k(K) \leq k(K^+)+k(K^-)$. 
\end{lemma}

Cut-dense graphs are in particular dense, as 
below.
\begin{lemma}[{\cite[Fact 3.6]{alexey}}]\label{pok3.6}
    Let $G$ be a $p$-cut-dense graph on $m\geq 10$ vertices. Then $e(G) \geq pm^2/4$. 
\end{lemma}
Not all dense graphs are cut-dense, of course, but they do contain large cut-dense subgraphs.
\begin{lemma}[{\cite[Lemma 3.14]{alexey}}]\label{pok3.14}
    Let $\nu \in (0,1)$ be sufficiently small, let $p$ be sufficiently small with respect to $\nu$, let $\mu$ be sufficiently small with respect to $p$, and let $s$ be a sufficiently large integer with respect to $\mu$. Let $G$ be an $s$-vertex graph with $e(G) \geq \nu s^2$. Then $G$ has a $p$-cut-dense subgraph with $\mu s$ vertices. 
\end{lemma}

Finally, we record that given a dense spot, we may find in it a dense regular subgraph, as formalized below; we note that stronger results were recently obtained in \cite{chakraborti_regular}, but the following result suffices for our purposes.

\begin{lemma}[{R\"odl and Wysocka \cite{rodl1997note}}]\label{pok4.5}
    Let $\eta \in (0,\frac 12)$, and let $r$ be sufficiently large with respect to $\eta$. Then for every integer $0 \leq x \leq 0.4 \eta^3 r$, every $r$-vertex graph $G$ with $e(G)\geq \eta r^2$ contains an $x$-regular subgraph.
\end{lemma}

\subsection{Proof of Theorem~\ref{thm: structure}}

Finally, we are ready to prove Theorem~\ref{thm: structure}. The dependencies of different parameters in Theorem~\ref{thm: structure} are already quite involved and in the proof we introduce several new parameters. To ease the reading, we demonstrate the relative magnitude of the parameters in Figure~\ref{fig: hierarchy}.

\begin{figure}[h]
    \centering
    \includegraphics[width=0.8\textwidth]{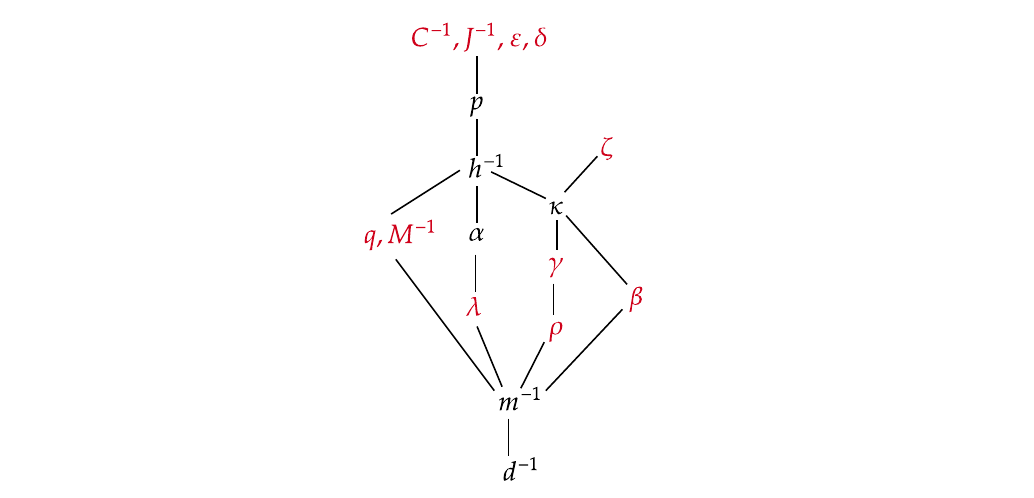}

    \caption{
    The diagram displays the partial order of parameters in the proof of Theorem~\ref{thm: structure}, where $x<y$ for two parameters indicates that $x$ is chosen to be small enough compared to $y$. Highlighted in red are the parameters present in the statement of Theorem~\ref{thm: structure}, whereas the new parameters in black are chosen to facilitate the proof. The two hierarchies are consistent except for the position of $\gamma$, which is assumed to be small in the proof, but this only makes the statement of Theorem~\ref{thm: structure} harder to prove.}
    \label{fig: hierarchy}
\end{figure}

\begin{proof}[Proof of Theorem~\ref{thm: structure}] Based only on the given values of $C,J,\varepsilon,\delta$, we choose new constants $p,h$ and $\alpha$ satisfying
$C^{-1},J^{-1},\varepsilon,\delta \gg p\gg  h^{-1}\gg \alpha$. Given $\zeta$, we pick $\kappa$ sufficiently small with respect to $\alpha$ and $\zeta$. Given $\gamma$, we pick $d$ sufficiently large with respect to all other parameters. Note that the final statement we are trying to prove only gets stronger if $\gamma$ is smaller, so we may assume that $\gamma \leq \kappa^{(10Chp^{-13})!}$. Finally, we set $m\coloneqq d/h$. We recall that we systematically omit floor and ceiling signs, so throughout the proof we will treat $m$ and related quantities, such as $p^{13}m$, as integers.

Although they will arise naturally later in the proof, let us explicitly record the final values of the parameters we will output, namely
\begin{align*}
\lambda \coloneqq \min \left\{\frac \alpha{10},\frac{p^3}{Ch}\right\}, \quad q \coloneqq \frac{p^{13}}{h}&, \quad M \coloneqq \max\left\{2Cp^{-3}h4^{Chp^{-13}},\frac{4^{2Chp^{-13}}}{h}\right\},\\  \rho \coloneqq \min\left\{\kappa^{(10\cdot2Chp^{-13})!}, \gamma^{h^4}\right\}&, \quad \beta\coloneqq\frac{\kappa^{(10Chp^{-13})!}p^{13}}{2h}.
\end{align*}
Note that, as claimed (recall the LHS of Figure~\ref{fig: hierarchy}), $\lambda,q,M$ depend only on $C,J,\delta,\varepsilon$, whereas $\beta$ also depends on $\zeta$ and $\rho$ depends on $\zeta$ as well as $\gamma$.
\par The choices of $p,h,\alpha, \kappa$ and $d$ are not made explicit. We only require them to be chosen so that the following conditions all hold, noting that these requirements are all consistent with the hierarchy as presented in Figure~\ref{fig: hierarchy}. 

\begin{itemize}
    \item $p \leq \min\{10^{-3}, \varepsilon/8\}$,
    \item $p$ is sufficiently small with respect to $J$ and $\delta$, as required by \cref{pok3.14} applied with $\nu=\delta/J^2$,
    \item $1/(\delta^{1/2}h)$ is sufficiently small with respect to $p$ so that we can apply \cref{pok3.14} with $\mu\leq1/(\delta^{1/2}h)$,
    \item $h\geq 2 C^3 p^{-19}$,
    \item $\alpha \leq \varepsilon p^{13}/(2h)$,
    \item $\kappa \leq p^{13}/(Ch)\leq p^{13}$,
    \item $\kappa \leq \zeta/h$,
    \item $\gamma \leq \kappa^{(10Chp^{-13})!}$, as stated above,
    \item $m$ is sufficiently large with respect to $p$ so that $r\coloneqq\sqrt{p/8}\cdot m$ is sufficiently large with respect to $\eta\coloneqq p/8$ to apply \cref{pok4.5}, 
    \item $m$ is sufficiently large with respect to $\kappa$ to apply \cref{pok4.8,pok4.9},
    \item $m$ is sufficiently large with respect to $\gamma$ to apply \cref{pok4.7}.
\end{itemize}

Having introduced all of the parameters, we now begin the proof in earnest.
 Pick a maximal collection of edge-disjoint graphs $H \subseteq G$ where $H$ is a $\rho$-cut-dense graph on at most $Md$ vertices containing a $qd$-regular subgraph of order at least $Cd$, and let $G_2$ be the union of this maximal collection. Next, pick a maximal collection of edge-disjoint $m$-vertex $p$-cut-dense subgraphs $K_1, \dots, K_t$ in $G\setminus G_2$. Let $G_{\mathrm{dense}}\coloneqq\bigcup_{i=1}^t K_i$, and set $G_3$ to be the graph on $V(G)$ with edge set $E(G)\setminus (E(G_2)\cup E(G_{\mathrm{dense}}))$.

\begin{claim}\label{claim:G1 good}
    $G_3$ does not contain a set of at most $Jd$ vertices inducing at least $\delta d^2$ edges.
\end{claim}
\begin{proof}
    Suppose not, and let $G'\subseteq G_3$ be the induced subgraph on these at most $Jd$ vertices. Set $s=|V(G')|$ and note that $\delta^{1/2} d\leq s\leq Jd $. We apply \cref{pok3.14} to $G'$ with parameters $\nu=\delta/J^2,p,\mu=m/s$, and $s$. By our choice of parameters, noting $\mu\leq 1/(\delta^{1/2}h)$, we have that $s^{-1}\ll \mu \ll p \ll \delta$, with all ``sufficiently small'' appropriate for applying \cref{pok3.14}. We conclude that $G'$ contains a $p$-cut-dense subgraph on $\mu s = m$ vertices, contradicting the maximality of the collection $K_1,\dots,K_t$.
\end{proof}

Next, we use \cref{pok4.8} to give each $K_i$ in the collection defining $G_{\mathrm{dense}}$ the structure of a non-empty clump with parameters $p, \kappa, m$, which we may do since $\kappa \leq p$ and $m$ is sufficiently large with respect to $\kappa$. For as long as possible, repeat the following: if we have two edge-disjoint clumps $K^-, K^+$ with $|M(K^-)|, |M(K^+)|\leq Cd$ 
and $|V(D(K^-))\cap V(D(K^+))|\geq \kappa^{6(10\max(k(K^-), k(K^+))!} m$, replace $K^-, K^+$ in our collection with the clump $K$ with $E(K)=E(K^-)\cup E(K^+)$ produced by \cref{pok4.9}. In order to apply \cref{pok4.9}, we need $m$ to be sufficiently large with respect to $\kappa$ (which it is by definition), and that $k(K^+),k(K^-)\leq \kappa^{-1}$. This latter holds because of \cref{pok4.3}, which gives 
\[
k(K^+) \leq \frac{\ab{V(M(K^+))}}{p^{13}m} \leq \frac{Cd}{p^{13}(d/h)} = \frac{Ch}{p^{13}}\leq \kappa^{-1},
\]
where the final step is by our choice of $\kappa \leq p^{13}/(Ch)$.

This process must eventually terminate, because the total number of clumps decreases at every step.
Let $\mathcal K$ be the final collection of clumps we get, noting that these are edge-disjoint, non-empty, and edge-partition $G_{\mathrm{dense}}$ (since the initial clumps had this property, and the property is maintained at each application of \cref{pok4.9}).

\begin{claim}\label{claim:M(Ki)_upper_bound}
$|V(M(K))|\leq Cd$ for all $K\in \mathcal K$.
\end{claim}
\begin{proof} 
First note that this holds for the initial clumps, since each $M(K_i)$ is contained in some $K_i$ which has order $m=d/h\leq Cd$. 
Now suppose for contradiction that at some point in the joining process, we produced a clump $K$ with $|V(M(K))|>Cd$. Considering the earliest point when this happens, we have that $E(K)=E(K^-)\cup E(K^+)$ for some clumps $K^-, K^+$ with $|V(M(K^-))|, |V(M(K^+))|\leq Cd$. By \cref{pok4.3}, we have $k(K^-), k(K^+)\leq \frac{Cd}{p^{13}m}=Chp^{-13}$. Hence, by \cref{pok4.9}, we have $k(K)\leq k(K^-)+k(K^+)\leq 2Chp^{-13}$. Thus, $C(K)\subseteq G_{\mathrm{dense}}$ is a $\kappa^{(10k(K))!}$-cut-dense graph of order at most $4^{k(K)}m$ containing $M(K)$, which is a $p^{13}m$-regular graph of order at least $Cd$. Together with $\rho \leq \kappa^{(10k(K))!}$, $p^{13}m=qd$, $4^{2Chp^{-13}}m\leq Md$ and $k(K)\leq 2Chp^{-13}$, this is in contradiction with the maximality of the collection defining $G_2$.
\end{proof}

This claim and \cref{pok4.3} show that for all clumps $K\in \mathcal K$, we have $Cd\geq |V(M(K))|\geq p^{13}mk(K)$, which implies  $k(K)\leq Cd/(p^{13}m)=Chp^{-13}$. By \cref{pok4.2}, we also see that each $B(K)$ has a cover of order at most $Cd$, namely $V(M(K))$.
 By the assumption that the process stopped and Claim~\ref{claim:M(Ki)_upper_bound}, we get that for distinct clumps $K, K'\in \mathcal K$ we have $$|V(D(K))\cap V(D(K'))|< \kappa^{6(10\max(k(K), k(K'))!} m\leq  \kappa m.$$  

\begin{claim}\label{claim:few big intersections}
    For every $K\in\mathcal K$ there are at most $Cp^{-3}h$ clumps $K'\in \mathcal K$ with $|V(D(K))\cap V(D(K'))|\geq \gamma m$.
\end{claim}
\begin{proof}
Suppose towards a contradiction there is a clump $K=K_1$ such that there are distinct clumps $K_2,\ldots,K_t$, where $t=Cp^{-3}h$, with $|V(D(K_1))\cap V(D(K_j))|\geq \gamma m$ for $1\leq j\leq t$. For $i=1, \dots, t$, pick arbitrary order-$m$, $p$-cut-dense $H_i\in \mathcal H(K_i)$, and pick arbitrary order-$\gamma m$ subsets $I_i\subseteq V(D(K_1))\cap V(D(K_i))$.
We now apply \cref{pok4.7} to this setup, with parameter $k=Chp^{-13}$, to conclude that the graph
\[
H \coloneqq \bigcup_{i=1}^t (D(K_i)[I_i \cup V(C(K_i))] \cup H_i)
\]
is cut-dense with cut-density parameter
\[
\kappa^{(10Chp^{-13})!}\gamma^{2Chp^{-13}t^2}\geq\gamma \cdot \gamma^{2Chp^{-13}(Cp^{-3}h)^2}\geq  \gamma^{h^4}\geq \rho,
\]
where we use our bounds $\gamma \leq \kappa^{(10Chp^{-13})!}$ and $h \geq 2C^3 p^{-19}$.

Furthermore, note that $|V(H)|\leq t\cdot (\gamma+4^{Chp^{-13}}+1)m\leq Md$, using the fact that $\ab{I_i}=\gamma m$, as well as \ref{pok(C)} and \ref{pok(H)}. 
We now claim that $H$ contains a $p^{13}m$-regular subgraph on more than $Cd$ vertices, which would contradict the maximality of the collection defining $G_2$.

To prove this, we first observe that
\begin{align*}
    e(H_i\cap D(K_i))&=e(H_i)-\ab{E(H_i) \setminus E(D(K_i))}\\
    &\geq e(H_i)-\ab{E(H_i) \setminus E(B(K_i))} &[B(K_i)\subseteq D(K_i)]\\
    &\geq \frac{pm^2}{4} - p^2 m^2 &[\text{\cref{pok3.6,pok4.6}}]\\
    &=\left(\frac p4 - p^2\right) m^2.
\end{align*}
Let $G'$ be a maximum-sized $p^{13}m$-regular subgraph of $\bigcup_{i=1}^{t} H_i$.
\begin{subclaim}
    $|V(G')\cap V(H_i\cap D(K_i))|\geq pm/20$ for all $i$. 
\end{subclaim}
\begin{proof}
Indeed, suppose otherwise that for some $i$, $|V(G')\cap V(H_i\cap D(K_i))|< pm/20$. 
Then 
\begin{align*}
    e((H_i\cap D(K_i))\setminus V(G'))&\geq e(H_i\cap D(K_i))-|V(G')\cap V(H_i\cap D(K_i))|\cdot|V(H_i)|\\
    &\geq \left(\frac p4-p^2\right)m^2-\frac{pm}{20}\cdot m\\
    &\geq \frac{pm^2}8\\
    &=\frac{p|V(H_i)|^2} 8\\
    &\geq \frac{p|V(H_i\cap D(K_i))\setminus V(G'))|^2}8.
\end{align*}
Set $r \coloneqq |V(H_i\cap D(K_i))\setminus V(G')|\geq \sqrt{e((H_i\cap D(K_i))\setminus V(G'))}\geq \sqrt{p/8}\cdot m$, and let $\eta=p/8$. Note too that $\eta \leq 1/2$ and that $r$ is, by definition, sufficiently large in terms of $p$ to apply \cref{pok4.5}. As $p^{13}m \leq 0.4 \eta^3 r$, we conclude from \cref{pok4.5} that 
$(H_i\cap D(K_i))\setminus V(G')$ contains a $p^{13}m$-regular subgraph, which by definition is vertex-disjoint from $G'$. Adding it to $G'$ contradicts the maximality of $G'$, and concludes the proof of the sub-claim.
\end{proof}
We may now apply inclusion-exclusion to find that
\begin{align*}
    \ab{V(G')} &\geq \sum_{i=1}^t \ab{V(G')\cap V(H_i \cap D(K_i))} - \sum_{1 \leq i<j\leq t} \ab{V(G') \cap V(H_i \cap D(K_i)) \cap V(H_j \cap D(K_j))}\\
    &\geq t \cdot \frac{pm}{20} - \sum_{1 \leq i<j\leq t} \ab{V(D(K_i)) \cap V(D(K_j))}\\
    &\geq t\cdot\frac{pm}{20}-t^2\cdot \kappa m\\
    &=m \cdot Ch\left(\frac{p^{-2}}{20} - Ch p^{-6}\kappa \right).
\end{align*}
Recalling that $Ch\kappa \leq p^{13}$ and $p \leq 10^{-3}$, we see that the parenthesized term is at least $2$. Hence $\ab{V(G')}\geq 2Chm = 2Cd$, contradicting the maximality of $G_2$. This contradiction concludes the proof of the claim. 
\end{proof}

Choose a maximal ordered subfamily $\mathcal K'=(K_1', \dots, K_r')$ of $\mathcal K$ with the property that $|V(D(K_i'))\cap (\bigcup_{K\in\mathcal K\setminus  \mathcal K'}V(D(K)) \cup \bigcup_{j>i} V(D(K_j')))|\leq \alpha d/10$ for all $i=1, \dots, r$, where we remark that $\mathcal K'$ may be empty. 

 Note that if for some $K\in \mathcal K\setminus \mathcal K'$, we have $|V(D(K))\cap \bigcup_{K'\in \mathcal K\setminus (\mathcal K'\cup \{K\})} V(D(K'))|\leq \alpha d/10$, then we could  define $K'_{r+1}\coloneqq K$ to get a bigger $\mathcal K'$, contradicting maximality. Set $\mathcal H = \{D(K) : K\in\mathcal K\setminus \mathcal K'\}$ and let $G_4$ be the union of the graphs in $\mathcal H$. We stress that $G_4$ is only the union of $D(K)$ over all $K \in \mathcal K \setminus \mathcal K'$, and in particular that $G_4$ does not contain all the edges in the clumps in $\mathcal K\setminus \mathcal K'$. Later we will count the number of edges lost here.

\begin{claim}\label{claim:G3 good}
    $G_4$ is the edge-disjoint union of a collection $\mathcal H$ of graphs satisfying the following properties.
        \begin{enumerate}
            \item Each $H\in\mathcal H$ contains a subgraph $Q$ of size at least $\beta d$ which is $\beta$-cut-dense and every vertex in $V(H)\setminus V(Q)$ has at least $\beta d$ neighbors in $V(Q)$.
            \item Any two $H,H'\in\mathcal H$ intersect in at most $\zeta d$ vertices.
            \item Every $H\in \mathcal H$ has at least $\lambda d$ vertices which intersect some other $H'\in\mathcal H$.
            \item For every $H\in\mathcal H$ there are at most $1/\lambda$ other $H'\in\mathcal H$ which intersect $H$ in more than $\gamma d$ vertices.
    \end{enumerate}
\end{claim}

\begin{proof}
As the clumps are edge-disjoint, we certainly have that all the graphs in the collection $\mathcal H$ are edge-disjoint, and by definition $G_4$ is the union of these graphs. For the first item, we observe that, by \ref{pok(M)} and \ref{pok(C)}, $C(K)$ contains at least $p^{13}m$ vertices and is $\kappa^{(10k(K))!}$-cut-dense. Further, $D(K)$ is the union of $C(K)$ and $B(K)$. Therefore, it is obtained from $C(K)$ by adding vertices with at least $pm/2$ neighbors in $V(C(K))$. Setting $Q=C(K)$ and using the definition of $\beta$, we get that $Q$ contains at least $\beta d$ vertices, is $\beta$-cut-dense and every vertex in $V(D(K))\setminus V(Q)$ has at least $\beta d$ neighbors in $V(Q)$.

Next, we recall that when the process terminates, we necessarily have that any pair of clumps $K,K' \in \mathcal K$ satisfy $\ab{V(D(K))\cap V(D(K'))}\leq \kappa m\leq \zeta d$, recalling our choice of $\kappa \leq \zeta/h$. Thus we satisfy the second item.

Next, recall from above that for all $K \in \mathcal K \setminus \mathcal K'$, we have $|V(D(K))\cap \bigcup_{K'\in \mathcal K\setminus (\mathcal K'\cup \{K\})} V(D(K'))|\geq \alpha d/10$. This proves the third item as we ensure $\lambda \leq \alpha/10$.

For the final item, we recall \cref{claim:few big intersections}, which states that for any $K \in \mathcal K$, there are at most $Cp^{-3}h$ other clumps $K' \in \mathcal K$ with $\ab{V(D(K))\cap V(D(K'))}\geq \gamma m$. In particular, since $m\leq d$, this implies the final item, as we chose $\lambda \leq p^3/(Ch)$.
\end{proof}
 
 We proceed. From each $K_i'\in\mathcal K'$ delete the edges $E(K_i')\setminus E(B(K_i'))$, the vertices  $V(K_i')\setminus V(B(K_i'))$, and the  vertices $V(D(K_i'))\cap (\bigcup_{j>i} V(D(K_j')))$  to obtain a subgraph $J_i$. Let $G_1 = \bigcup_{i=1}^r J_i$. 
 \begin{claim}\label{claim:G4 good}
     Every connected component of $G_1$ has a cover of size at most $Cd$.
 \end{claim}
 \begin{proof}
 We first claim that all the $J_i$'s are vertex-disjoint. Indeed, fix some $i<j$, and note that $V(J_j)\subseteq V(B(K_j'))\subseteq  V(D(K_j'))$, which implies that all vertices of $J_j$ were deleted from $K_i'$ when forming $J_i$, so they are vertex-disjoint as claimed. 
Hence $G_1$ is a vertex-disjoint union of $J_1,\dots,J_r$, implying that every connected component of $G_1$ is contained in some $J_i$.
 
Since each $J_i\subseteq B(K_i')$  has a cover of size at most $ Cd$ (namely $V(M(K_i'))$ by \cref{pok4.2}), the connected components of $G_1$ have covers of size at most $Cd$.  
\end{proof}
Finally, we let $G_0 = G \setminus (G_1\cup G_2 \cup G_3 \cup G_4)$, and we now turn to estimating $e(G_0)$, or equivalently the number of edges of $G$ absent from $G_1\cup G_2\cup G_3\cup G_4$. Every such edge is either in $E(K_i')\setminus E(J_i)$ for some $1 \leq i \leq r$, or in $E(K) \setminus E(D(K))$ for some $K \in \mathcal K \setminus K'$. Note that every edge in $E(K_i') \setminus E(J_i)$ is either contained in $E(K_i') \setminus E(B(K_i'))$, or else it is an edge in $B(K_i')$, one of whose endpoints is in $V(D(K_i'))\cap (\bigcup_{j>i}V(D(K_j')))$. The total number of edges of this second type is at most
\begin{align*}
    \sum_{i=1}^r \ab{V(B(K_i'))}\cdot \bab{V(D(K_i'))\cap \left(\bigcup_{j>i}V(D(K_j'))\right)} & \leq \sum_{i=1}^r \ab{V(B(K_i'))}\cdot \frac{\alpha d}{10}\\
    &\leq \sum_{i=1}^r \frac{\alpha d} {p^{13}m} e(K_i') &[\text{\cref{pok4.4}}]\\
    &=\alpha hp^{-13} \sum_{i=1}^r e(K_i')\\
    &\leq \alpha hp^{-13}e(G).
\end{align*}
We now note that all the other edges of $G$ which are not in $G_1 \cup G_2 \cup G_3 \cup G_4$ are in $E(K) \setminus E(B(K))$ for some $K \in \mathcal K$. For the clumps $K \in \mathcal K'$ this is automatic, and for the clumps $K \in \mathcal K\setminus \mathcal K'$ it follows from the fact that $D(K)\supseteq B(K)$, and hence the set $E(K) \setminus E(D(K))$ of missing edges is contained in the set $E(K) \setminus E(B(K))$. Thus, the total number of remaining edges we deleted is at most
\[
\sum_{K \in \mathcal K}\ab{E(K)\setminus E(B(K))} \leq \sum_{K \in \mathcal K}4p e(K) = 4p\sum_{K \in \mathcal K} e(K) \leq 4pe(G),
\]
where we use \cref{pok4.6} in the first inequality. In total, we find that
\[
e(G_0)\leq (\alpha h p^{-13}+4p)e(G) \leq \left(\frac \varepsilon 2 +\frac \varepsilon 2\right) e(G)= \varepsilon e(G),
\]
as claimed, which concludes the proof.
\end{proof}

\end{document}